\documentclass[USenglish,onecolumn]{article}

\usepackage[utf8]{inputenc}
\usepackage[big]{dgruyter}

\usepackage{amsfonts}
\usepackage{amsmath}
\usepackage{amsthm}
\usepackage{sidecap}
\usepackage[mathscr]{euscript}
\DeclareMathOperator{\diam}{diam}
\DeclareMathOperator*{\esup}{ess\,sup}
\newcommand{\fA}{{\mathfrak A}}
\newcommand{\sG}{\mathscr{G}}
\newcommand{\sH}{\mathscr{H}}
\newcommand{\sI}{\mathscr{I}}
\newcommand{\sK}{{\mathscr K}}
\newcommand{\sU}{{\mathscr U}}
\newcommand{\K}{{\mathbb K}}
\newcommand{\N}{{\mathbb N}}
\newcommand{\R}{{\mathbb R}}

\renewcommand{\d}{\,{\mathrm d}}
\newcommand{\eps}{\varepsilon}
\newcommand{\tm}{\times}
\newcommand{\intoo}[1]{\left(#1\right)}				
\newcommand{\intcc}[1]{\left[#1\right]}				
\newcommand{\set}[1]{\left\{#1\right\}}				
\newcommand{\abs}[1]{\left|#1\right|}				
\newcommand{\norm}[1]{\left\|#1\right\|}			
\newcommand{\fall}{\quad\text{for all }}

\newtheorem*{hypo}{Hypothesis}
\theoremstyle{definition}

\theoremstyle{remark}
\newtheorem{rem}{Remark}[section]
\newtheorem{ex}{Example}[section]
\theoremstyle{plain}
\newtheorem{thm}{Theorem}[section]
\newtheorem{cor}[thm]{Corollary}
\newtheorem{lem}[thm]{Lemma}
\newtheorem{prop}[thm]{Proposition}
\newcommand{\cref}[1]{Cor.~\ref{#1}}

\newcommand{\eref}[1]{Ex.~\ref{#1}}

\newcommand{\fref}[1]{Fig.~\ref{#1}}
\newcommand{\lref}[1]{Lemma~\ref{#1}}
\newcommand{\pref}[1]{Prop.~\ref{#1}}
\newcommand{\rref}[1]{Rem.~\ref{#1}}
\newcommand{\sref}[1]{Sect.~\ref{#1}}
\newcommand{\tref}[1]{Thm.~\ref{#1}}
\numberwithin{equation}{section}
\begin{document}
\articletype{Research Article{\hfill}}
\author[1]{Christian P\"otzsche}
\affil[1]{Institut f\"ur Mathematik, Universit\"at Klagenfurt, Universit\"atsstra{\ss}e 65--67, 9020 Klagenfurt, Austria, E-mail: christian.poetzsche@aau.at}
\title{\huge Urysohn and Hammerstein operators\\ on H\"older spaces}
\runningtitle{Urysohn and Hammerstein operators on H\"older spaces}
\begin{abstract}
	{We present an application-oriented approach to Urysohn and Hammerstein integral operators acting between spaces of H\"older continuous functions over compact metric spaces. These nonlinear mappings are formulated by means of an abstract measure theoretical integral involving a finite measure. This flexible setting creates a common framework to tackle both such operators based on the Lebesgue integral like frequently met in applications, as well as e.g.\ their spatial discretization using stable quadrature/cubature rules (Nystr\"om methods). Under suitable Carath{\'e}odory conditions on the kernel functions, properties like well-definedness, boundedness, (complete) continuity and continuous differentiability are established. Furthermore, the special case of Hammerstein operators is understood as composition of Fredholm and Nemytskii operators. While our differentiability results for Urysohn operators appear to be new, the section on Nemytskii operators has a survey character. Finally, an appendix provides a rather comprehensive account summarizing the required preliminaries for H\"older continuous functions defined on metric spaces.}
\end{abstract}

\DOI{DOI}
\startpage{1}
\received{\today}
\revised{..}
\accepted{..}

\journalyear{2021}
\journalvolume{1}
%
\maketitle
\section{Introduction}
This treatise is devoted to Urysohn operators, a class of nonlinear integral operators arising in various contexts of nonlinear analysis \cite{bardaro:etal:03,khatskevich:shoiykhet:94,martin:76,pathak:18,precup:02}, as right-hand sides of certain integrodifferential (Barbashin) equations \cite{appell:kalitvin:zabrejko:00}, as well as in recent applications from control theory \cite{huseyin:guseinov:18}, mathematical biology \cite{amar:jeribi:mnif:08}, economic theory \cite{edmond:08} (integral over an unbounded domain) or system identification \cite{poluektov:polar:20} (sums as integrals). Urysohn operators are traditionally well-studied when acting between spaces of continuous functions over a compact domain \cite[pp.~164ff, Sect.~V.3]{martin:76}, \cite[pp.~35--37, Sect.~3.1]{precup:02} or \cite[App.~B.2]{poetzsche:18a}, spaces of integrable functions \cite{krasnoselskii:etal:76} with possibly different exponents, or in mixed form \cite[pp.~175ff]{martin:76}. In such a set-up, their well-definedness and continuity is addressed e.g.\ in \cite[pp.~172ff]{martin:76}, \cite[p.~85]{pathak:18}, while conditions yielding that they are set contractions w.r.t.\ ambient measures of non-compactness can be found in \cite[pp.~227ff]{akhmerov:etal:92} (for $L^p$-spaces). Both necessary and sufficient conditions for the complete continuity of Urysohn operators between different function spaces are given in \cite{misyurkeev:91}. Furthermore, \cite[pp.~162--298]{vaeth:00} provides an extensive analysis of such mappings between abstract ideal spaces; see also \cite{vaeth:97}. Properties of Urysohn operators over compact intervals having values in a real Banach space are discussed in \cite[pp.~54--90, Sect.~2.1]{guo:lakshmikantham:liu:96} and differentiability conditions were given in \cite{durdil:67} (see also \cite[pp.~41ff]{khatskevich:shoiykhet:94} or \cite[pp.~417ff, Sect.~20]{krasnoselskii:etal:76} in $L^p$-spaces). Finally, we would like to point out the paper \cite{krukowski:przeradzki:16} containing complete continuity results for Urysohn operators on the continuous functions over merely locally compact (and possibly unbounded) domains. 

A highly relevant special case is given in terms of Hammerstein operators \cite{khatskevich:shoiykhet:94,krasnoselskii:etal:76,martin:76,pathak:18}. Our given approach tackles them as composition of (linear) Fredholm integral operators determined by an integral kernel \cite{fenyo:stolle:82,hackbusch:95,kress:14} with (nonlinear) Nemytskii operators \cite{appell:zabrejko:90,chiappinelli:nugari:95,drabek:75,goebel:sachweh:99,matkowska:84,matkowski:09,nugari:93,pathak:18}. Classically Hammerstein operators arise in fixed point problems related to nonlinear boundary value problems \cite[pp.~177ff, Sect.~V.5]{martin:76} or \cite[p.~71, Thm.~5.5]{precup:02}, where the kernel is a corresponding Green's function. A more recent application are integrodifference equations originating in theoretical ecology \cite{kot:schaeffer:86,lutscher:19}, where the kernel models the spatial dispersal of species over a habitat, while the Nemytskii operator describes their growth phase. The classical $L^p$-theory of Hammerstein operators is covered for instance in \cite[p.~84]{pathak:18} or \cite[pp.~68ff, Sect.~5.3]{precup:02}. 

The paper at hand supplements the above contributions. We provide a comprehensive approach to Urysohn operators acting between possibly different spaces of H\"older continuous functions over compact metric spaces. We restrict to H\"older spaces with exponents $\alpha\leq 1$, i.e.\ the functions under consideration are not necessarily differentiable with H\"older continuous derivatives of positive order. This endows us with a wide scale of Banach spaces whose elements range from nowhere differentiable to Lipschitz functions, being differentiable almost everywhere. On the one hand, an early contribution to this area is the note \cite{pachale:59} addressing well-definedness and complete continuity of general Urysohn operators. On the other hand, H\"older spaces are meanwhile widely used when dealing with linear integral operators having singular kernels \cite[pp.~103ff, Ch.~7]{kress:14} or in the field of (quasilinear) elliptic boundary value problems \cite{gilbarg:trudinger:01}; moreover, \cite{saiedinezhad:18} investigates nonlinear integral equations in H\"older spaces. Our motivation, nevertheless, is different. It rather comes from the numerical analysis of integral equations \cite{atkinson:92} and the numerical dynamics of integrodifference equations \cite{kot:schaeffer:86,lutscher:19,poetzsche:18a}. In the latter context one aims to show that such infinite-dimensional dynamical systems given by the iterates of integral operators share the long term dynamics with their spatial discretizations. Certain problems in this area require to establish that Fr{\'e}chet derivatives of Urysohn operators and of their spatial discretization converge to each other in the operator norm. Among the techniques for the numerical solution of integral equations this can be justified for semi-discretizations of projection or degenerate kernel type, cf.\ \cite{atkinson:92,hackbusch:95,kress:14} and \cite{poetzsche:18a}. However, uniform convergence is not feasible when working with full discretizations of Nystr\"om type on the continuous functions (see e.g.\ \cite[p.~225, Thm.~12.8]{kress:14}). In contrast, when working with H\"older spaces appropriate estimates can be established \cite{poetzsche:20}. 

Having applications from theoretical ecology to numerical dynamics in mind, it is advantageous to establish a rather flexible setting we are aiming to provide here: First, we consider vector-valued operators (in finite dimensions though), which arise in ecological models describing various interacting species. Second, we allow general measure theoretical integrals induced by a finite measure such that both integral operators based on the Lebesgue integral, as well as their spatial discretization using e.g.\ Nystr\"om methods fit into a common framework (see \eref{exnyst}). For this reason we content ourselves to provide sufficient conditions guaranteeing that an integral operator is well-defined, bounded, (completely) continuous or differentiable. Necessary conditions for the above properties exist for operators on compact intervals, but are beyond the scope of this paper. 

Our presentation is subdivided into three parts: In \sref{sec2} we provide conditions of Carath{\'e}odory type on the kernel functions such that the associated Urysohn operators are well-defined, bounded, (completely and H\"older) continuous, resp.\ continuously differentiable. We successively study such operators, first having values in the continuous, and second in the H\"older functions. In particular, a subsection is devoted to convolutive Urysohn operators $\tilde\sU$, where H\"older continuity of the arguments $u$ extends to the values $\tilde\sU(u)$. The \sref{sec3} on Hammerstein operators follows a similar scheme. These mappings are compositions of Fredholm and Nemytskii operators. Since Nemytskii operators between H\"older spaces have rather degenerate mapping and differentiability properties \cite[Ch.~7]{appell:zabrejko:90}, we retreat to the case that they map into the continuous functions. H\"older continuity of the images is then guaranteed by appropriate assumptions on the kernel of the subsequent Fredholm operator. Addressing well-studied objects, the \sref{sec32} on Nemytskii operators between H\"older spaces has a survey character. Finally, the App.~\ref{appA} provides a broad perspective over the class of H\"older continuous functions defined on a metric space and having values in a normed space. 

\textbf{Notation and terminology}: Let $\R_+:=[0,\infty)$ and $X,Y$ be nonempty sets. We write $F(X,Y)$ for the set of all functions $f:X\to Y$. In the setting of metric spaces $X,Y$, a subset $\Omega\subseteq X$ is called \emph{bounded}, if it has finite \emph{diameter}
\begin{displaymath}
	\diam\Omega:=\sup_{x,\bar x\in\Omega}d(x,\bar x).
\end{displaymath}
A function $f:X\to Y$ is called \emph{bounded}, it if maps bounded sets into bounded sets, i.e.\ $f(\Omega)\subseteq Y$ is bounded for every bounded $\Omega\subseteq X$, and \emph{globally bounded}, if $f(X)\subseteq Y$ is bounded. A \emph{completely continuous} mapping is continuous and maps bounded sets into relatively compact images. 

If $X,Y$ are normed spaces, then $L_k(X,Y)$, $k\in\N_0$, is the normed space of continuous $k$-linear maps from $X^k$ to $Y$, where $L_0(X,Y):=Y$ and $L(X,Y):=L_1(X,Y)$. We write $B_r(x_0,X):=\set{x\in X:\,\norm{x-x_0}<r}$ for the open and $\bar B_r(x_0,X):=\set{x\in X:\,\norm{x-x_0}\leq r}$ for the closed $r$-ball around $x_0\in X$ in $(X,\norm{\cdot})$. Norms on finite-dimensional spaces are denoted as $\abs{\cdot}$ and $B_r(x_0)$, $\bar B_r(x_0)$ are the corresponding $r$-balls. 

The remaining introduction anticipates notions from App.~\ref{appA} on H\"older spaces: A function $u:\Omega\to\R^n$ on a metric space $(\Omega,d)$ is called \emph{$\alpha$-H\"older} (with \emph{H\"older exponent} $\alpha\in(0,1]$), if it satisfies
\begin{displaymath}
	[u]_\alpha:=\sup_{\substack{x,\bar x\in\Omega\\ x\neq\bar x}}
	\frac{\abs{u(x)-u(\bar x)}}{d(x,\bar x)^\alpha}<\infty;
\end{displaymath}
the finite quantity $[u]_\alpha$ is denoted as \emph{H\"older constant} of $u$. One speaks of a \emph{H\"older continuous} function $u$, if it is $\alpha$-H\"older for some $\alpha\in(0,1)$, in case $\alpha=1$ one denotes $u$ as \emph{Lipschitz continuous} with \emph{Lipschitz constant} $[u]_1$ and for convenience we denote a continuous function also as $0$-H\"older. For the linear space of all bounded and $\alpha$-H\"older functions we write\footnote{note that $C_n^1(\Omega)=C^1(\Omega,\R^n)$ abbreviates the Lipschitz continuous and not the continuously differentiable functions} 
$
	C_n^\alpha(\Omega):=C^\alpha(\Omega,\R^n),
$
supplemented by $C_n^0(\Omega):=C^0(\Omega,\R^n)$ for the bounded, continuous functions and $C^\alpha(\Omega):=C^\alpha(\Omega,\R)$. Note that $C_n^\alpha(\Omega)$ is a Banach space w.r.t.\ the norm (cf.~\tref{thmcomplete})
\begin{displaymath}
	\norm{u}_\alpha:=
	\begin{cases}
		\sup_{x\in\Omega}\abs{u(x)},&\alpha=0,\\
		\max\set{\sup_{x\in\Omega}\abs{u(x)},[u]_\alpha},&\alpha\in(0,1].
	\end{cases}
\end{displaymath}

Throughout the remaining paper, our set-up is as follows: Let $\Omega$ and $\Omega_1$ be metric spaces. Suppose additionally that $\Omega$ is compact and can be interpreted as measure space $(\Omega,\fA,\mu)$ with $\mu(\Omega)<\infty$ whose $\sigma$-algebra $\fA$ contains the Borel sets. The notions of measurability and integrability refer to this measure space from now on. In particular, $\int_\Omega\cdot\d\mu$ stands for the abstract integral associated to the measure $\mu$ (e.g.\ \cite{cohn:80}). 

Moreover, $Z\subseteq\R^n$ denotes a nonempty subset throughout. Given a H\"older exponent $\alpha\in[0,1]$ we write
\begin{displaymath}
	U_\alpha:=\set{u:\Omega\to Z\mid u\in C_n^\alpha(\Omega)}
\end{displaymath}
for the $\alpha$-H\"older functions over $\Omega$ having values in $Z$. If $0\leq\alpha\leq\beta\leq 1$, then \tref{thmembed} guarantees the embedding $U_\beta\subseteq U_\alpha$ between the continuous and the Lipschitz continuous functions. 
\section{Urysohn integral operators}
\label{sec2}

An \emph{Urysohn operator}\footnote{also denoted as \emph{nonlinear Fredholm operator}. Another transcription is \emph{Uryson} operator} is a nonlinear integral operator of the form
\begin{align}
	\sU:U_\alpha&\to F(\Omega_1,\R^d),&
	\sU(u)&:=\int_\Omega f(\cdot,y,u(y))\d\mu(y)
	\label{udef}
\end{align}
determined by a \emph{kernel function} $f:\Omega_1\tm\Omega\tm Z\to\R^d$ and a measure $\mu$ as above. Its overall analysis is based on the following Carath{\'e}odory like conditions: 
\begin{hypo}
	Let $m\in\N_0$. With $0\leq k\leq m$ one assumes: 
	\begin{itemize}
		\item[$(U_0^k)$] The partial derivative $D_3^kf(x,y,\cdot):Z\to L_k(\R^n,\R^d)$ exists and is continuous for all $x\in\Omega_1$ and almost all $y\in\Omega$, 

		\item[$(U_1^k)$] for all $r>0$ there exists a function $h_r^k:\Omega_1^2\tm\Omega\to\R_+$, measurable in the third argument and satisfying 
		\begin{equation}
			\lim_{x\to x_0}\int_\Omega h_r^k(x,x_0,y)\d\mu(y)=0\fall x_0\in\Omega_1, 
			\label{ua00}
		\end{equation}
		so that for almost all $y\in\Omega$ the following holds: 
		\begin{align}
			\abs{D_3^kf(x,y,z)-D_3^kf(x_0,y,z)}&\leq h_r^k(x,x_0,y)
			\fall x,x_0\in\Omega_1,\,z\in Z\cap\bar B_r(0), 
			\label{ua01}
		\end{align}

		\item[$(U_2^k)$] $D_3^kf(x,\cdot,z):\Omega\to L_k(\R^n,\R^d)$ is measurable for all $x\in\Omega_1$, $z\in Z$, and suppose that for every $r>0$ there exists a function $b_r^k:\Omega_1\tm\Omega\to\R_+$ measurable in the second argument and satisfying $\esup_{\xi\in\Omega_1}\int_\Omega b_r^k(\xi,y)\d\mu(y)<\infty$, so that for almost all $y\in\Omega$ the following holds:
		\begin{align}
			\abs{D_3^kf(x,y,z)}&\leq b_r^k(x,y)
			\fall x\in\Omega_1,\,z\in Z\cap\bar B_r(0). 
			\label{ua02}
		\end{align}
	\end{itemize}
\end{hypo}
Because we are working with a general (finite) measure on $\Omega$, both spatially continuous and discrete integral operators fit into our framework: 
\begin{ex}[Lebesgue measure]\label{exleb}
	In most applications, e.g.\ \cite{amar:jeribi:mnif:08,edmond:08,kot:schaeffer:86,lutscher:19}, $\mu$ is the $\kappa$-dimensional Lebesgue measure $\lambda_\kappa$ on compact sets $\Omega\subset\R^\kappa$ yielding the Lebesgue integral in \eqref{udef} and thus
	\begin{equation}
		\sU(u)
		=
		\int_\Omega f(\cdot,y,u(y))\d\lambda_\kappa(y)
		=
		\int_\Omega f(\cdot,y,u(y))\d y:\Omega_1\to\R^d
		\label{nouryleb}
	\end{equation}
	is a spatially continuous integral operator. One clearly has $\mu(\Omega)<\infty$. 
\end{ex}
\begin{ex}[Nystr\"om methods]\label{exnyst}
	Suppose that $\Omega\subset\R^\kappa$ is a countable set $\Omega^{(l)}$, $\eta\in\Omega^{(l)}$ and $w_\eta$ denote non\-negative reals. Then $\mu(\Omega^{(l)}):=\sum_{\eta\in\Omega^{(l)}}w_\eta$ defines a measure on the family of countable subsets of $\R^\kappa$ and precisely the empty set has measure $0$. Moreover, the assumption $\sum_{\eta\in\Omega^{(l)}}w_\eta<\infty$ guarantees that $\mu(\Omega^{(l)})$ is finite. The resulting $\mu$-integral $\int_\Omega u\d\mu=\sum_{\eta\in\Omega^{(l)}}w_\eta u(\eta)$ leads to spatially discrete Urysohn operators
	\begin{displaymath}
		\sU(u)
		=
		\int_{\Omega^{(l)}}f(\cdot,y,u(y))\d\mu(y)
		=
		\sum_{\eta\in\Omega^{(l)}}w_\eta f(\cdot,\eta,u(\eta)):\Omega_1\to\R^d, 
	\end{displaymath}
	which cover \emph{Nystr\"om methods} with \emph{nodes} $\eta$ and \emph{weights} $w_\eta$ as used for numerical approximations of spatially continuous integral operators \eqref{nouryleb}, cf.\ \cite[Sect.~3]{atkinson:92}, \cite[pp.~128ff, Sect.~4.7]{hackbusch:95} or \cite[pp.~219ff, Ch.~12]{kress:14} (the latter two references address linear operators only). Alternatively, such mappings arise in theoretical ecology by means of models for populations spreading between finitely many different patches (\emph{metapopulation models}, see \cite[Example~1]{kot:schaeffer:86}). 
\end{ex}
\begin{ex}[evaluation map]
	In case of singletons $\Omega=\set{\eta}$ and the measure from \eref{exnyst} one obtains that $F(\Omega,\R^n)\cong\R^n$ and the Urysohn operator \eqref{udef} becomes an evaluation map $\sU(u)=w_\eta f(\cdot,\eta,u(\eta))$, which is simply a mapping from $\R^n$ into $F(\Omega_1,\R^d)$. 
\end{ex}

\begin{rem}[differentiability on $Z$]
	We imposed no further conditions of the sets $Z\subseteq\R^n$ and therefore some remarks on the existence of the partial derivative $D_3^kf$ for $k>0$ are due:
	
	(1) For interior points of $Z$ the partial derivatives are understood in the Fr{\'e}chet sense. 

	(2) If $z_0\in Z$ is not an interior point of $Z$, then we assume that there exists a neighborhood $V\subseteq\R^n$ of $z_0$ and an extension $\bar f:\Omega_1\tm\Omega\tm(Z\cup V)\to\R^d$ such that the partial derivatives $D_3\bar f(x,y,\cdot)$ exist in $z_0$ as assumed in $(U_0^k)$. Alternatively, there is the notion of cone differentiability \cite[pp.~225--226]{deimling:85}. 
\end{rem}

Under continuity the above hypothesis can be simplified as follows:
\begin{prop}\label{propcont}
	Let $k\in\N_0$, $\Omega_1$ be compact and $Z\subseteq\R^n$ be closed. If the partial derivative $D_3^kf:\Omega_1\tm\Omega\tm Z\to L_k(\R^n,\R^d)$ exists as continuous function, then $(U_0^k,U_1^k,U_2^k)$ are satisfied and the limit relation \eqref{ua00} holds uniformly in $x_0\in\Omega_1$. 
\end{prop}
\begin{proof}
	Since $Z$ is assumed to be closed, $Z_r:=Z\cap\bar B_r(0)\subseteq\R^n$ is compact for all $r>0$. Then the continuous function $D_3^kf$ is uniformly continuous and globally bounded on each compact product $\Omega_1\tm\Omega\tm Z_r$. Moreover, since the $\sigma$-algebra $\fA$ contains the Borel sets, continuous functions are measurable. Given this, the assertions hold with the continuous functions
	\begin{align*}
		h_r^k(x,x_0,y)&:=\sup_{z\in Z_r}\abs{D_3^kf(x,y,z)-D_3^kf(x_0,y,z)},&
		b_r^k(x,y)&:=\sup_{z\in Z_r}\abs{D_3^kf(x,y,z)}. 
	\end{align*}
	This concludes the proof. 
\end{proof}
\subsection{Well-definedness and complete continuity}
We begin with basic properties of Urysohn operators \eqref{udef} and assume $\alpha\in(0,1]$: 
\begin{prop}[well-definedness of $\sU$]\label{propuwell}
	Assume that $(U_0^0,U_1^0,U_2^0)$ hold. Then an Urysohn operator $\sU:U_\alpha\to C_d^0(\Omega_1)$ is well-defined, bounded and continuous. 
\end{prop}
\begin{proof}
	W.l.o.g.\ let $\mu(\Omega)>0$ since otherwise $\sU(u)\equiv 0$ on $U_\alpha$. 

	(I) Claim: \emph{$\sU:U_\alpha\to C_d^0(\Omega_1)$ is well-defined and bounded.}\\
	Choose $u\in U_\alpha$ and $r>0$ such that $\norm{u}_0\leq r$. Given $x,x_0\in\Omega_1$ the Carath{\'e}odory conditions $(U_0^0,U_2^0)$ yield that $f(x,\cdot,u(\cdot)):\Omega\to\R^d$ is measurable (see \cite[p.~62, Lemma~5.1]{precup:02}). From $(U_1^0)$ we conclude
	\begin{displaymath}
		\abs{\sU(u)(x)-\sU(u)(x_0)}
		\stackrel{\eqref{udef}}{\leq}
		\int_\Omega\abs{f(x,y,u(y))-f(x_0,y,u(y))}\d\mu(y)
		\stackrel{\eqref{ua01}}{\leq}
		\int_\Omega h_r^0(x,x_0,y)\d\mu(y)
		\xrightarrow[x\to x_0]{\eqref{ua00}}0
	\end{displaymath}
	for each $x_0\in\Omega_1$, which guarantees that $\sU(u)$ is continuous. Furthermore, because $(U_2^0)$ yields
	\begin{displaymath}
		\abs{\sU(u)(x)}
		\stackrel{\eqref{udef}}{\leq}
		\int_\Omega\abs{f(x,y,u(y))}\d\mu(y)
		\stackrel{\eqref{ua02}}{\leq}
		\esup_{\xi\in\Omega_1}\int_\Omega b_r^0(\xi,y)\d\mu(y)\fall x\in\Omega_1
	\end{displaymath}
	we see that $\sU(u)$ is bounded and thus $\sU(u)\in C_d^0(\Omega_1)$. In addition, $\sU$ maps bounded subsets of $U_0$ into bounded subsets of $C_d^0(\Omega_1)$. 

	(II) Claim: \emph{$\sU:U_\alpha\to C_d^0(\Omega_1)$ is continuous.}\\
	Let $u\in U_\alpha$ and $(u_l)_{l\in\N}$ be a sequence in $U_\alpha$ with $\lim_{l\to\infty}\norm{u_l-u}_0=0$ and $r>0$ sufficiently large so that $u,u_l\in\bar B_r(0,C_n^0(\Omega))$ holds for all $l\in\N$. Using $(U_0^0)$ this gives $\lim_{l\to\infty}f(x,y,u_l(y))=f(x,y,u(y))$ for all $x\in\Omega_1$ and almost all $y\in\Omega$. For each $\eps>0$ there exists a $\delta>0$ such that subsets $\tilde\Omega\subseteq\Omega$ with $\mu(\tilde\Omega)\leq\delta$ fulfill $\int_{\tilde\Omega}b_r^0(x,y)\d\mu(y)\leq\tfrac{\eps}{4}$ and Egoroff's theorem \cite[p.~87, Prop.~3.1.3]{cohn:80} guarantees that there exist a $\Omega'\subseteq\Omega$ with $\mu(\Omega')\leq\delta$ and an $L\in\N$ such that
	$
		\abs{f(x,y,u_l(y))-f(x,y,u(y))}\leq\tfrac{\eps}{2\mu(\Omega)}
	$
	for all $x\in\Omega_1$, $y\in\Omega\setminus\Omega'$ and $l\geq L$. This implies that we have pointwise convergence due to
	\begin{align*}
		\abs{[\sU(u_l)-\sU(u)](x)}
		&\stackrel{\eqref{udef}}{\leq}
		\int_{\Omega\setminus\Omega'}\abs{f(x,y,u_l(y))-f(x,y,u(y))}\d\mu(y)+
		\int_{\Omega'}\abs{f(x,y,u_l(y))-f(x,y,u(y))}\d\mu(y)\\
		&\stackrel{\eqref{ua02}}{\leq}
		\int_{\Omega\setminus\Omega'}\tfrac{\eps}{2\mu(\Omega)}\d\mu(y)+
		2\int_{\Omega'}b_r^0(x,y)\d\mu(y)
		\leq
		\tfrac{\eps}{2}+2\tfrac{\eps}{4}=\eps\fall x\in\Omega_1,\,l\geq L. 
	\end{align*}
	Passing to the supremum over $x\in\Omega_1$ yields the limit relation
	\begin{equation}
		\lim_{l\to\infty}\norm{\sU(u_l)-\sU(u)}_0=0.
		\label{ucont0}
	\end{equation}
	This shows the continuity of $\sU$. 
\end{proof}
\begin{cor}[complete continuity of $\sU$]\label{corcc}
	An Urysohn operator $\sU:U_\alpha\to C_d^0(\Omega_1)$ is completely continuous, provided one of the following holds:
	\begin{itemize}
		\item[(i)] $\alpha\in(0,1]$, 

		\item[(ii)] $\Omega_1$ is compact, $\alpha=0$, the limit relation \eqref{ua00} holds uniformly in $x_0\in\Omega_1$.
	\end{itemize}
\end{cor}
\begin{proof}
	We write $\sU_0:U_0\to C_d^0(\Omega_1)$ for the operator defined in \eqref{udef}. 
	
	(I) For exponents $\alpha\in(0,1]$ we observe $\sU=\sU_0\circ\sI_\alpha^0$ with the compact embedding operator $\sI_\alpha^0$ from \eqref{noe} (cf.~\tref{thm10}). Therefore, $\sU$ inherits the claimed properties from the steps (I) and (II) of the proof to \pref{propuwell}. In particular, as composition of the continuous $\sU_0$ with the compact $\sI_\alpha^0$ it is completely continuous due to e.g.\ \cite[pp.~25--26, Thm.~2.1(2)]{precup:02}. 

	(II) Claim: \emph{If (ii) holds, then $\sU_0:U_0\to C_d^0(\Omega_1)$ is completely continuous.}\\
	In \pref{propuwell} it was shown that $\sU_0:U_0\to C_d^0(\Omega_1)$ is bounded and continuous. If $\alpha=0$ and \eqref{ua00} holds uniformly in $x_0\in\Omega_1$, then the first limit relation in the proof of \pref{propuwell}) is true uniformly in $x_0$ as well, and each image $\sU_0\bigl(U_0\cap\bar B_r(0,C_n^0(\Omega))\bigr)\subset C_d^0(\Omega_1)$ is equicontinuous. Therefore, the Arzel{\`a}-Ascoli theorem \cite[p.~31, Thm.~3.2]{martin:76} yield its relative compactness. 
\end{proof}

\begin{cor}
	Let $\Omega_1$ be compact and $Z\subseteq\R^n$ be closed. If $f:\Omega_1\tm\Omega\tm Z\to\R^d$ is continuous, then $\sU:U_\alpha\to C_d^0(\Omega_1)$ is completely continuous and uniformly continuous on each set $U_\alpha\cap\bar B_r(0,C_n^0(\Omega))$, $r>0$. 
\end{cor}
\begin{proof}
	Due to \pref{propcont} the assumptions of \pref{propuwell} and \cref{corcc} are fulfilled. Therefore, $\sU$ is completely continuous. Hence, it remains to show the uniform continuity on bounded sets. For this purpose, let $\eps>0$ and $r>0$. Because $Z$ is closed, $Z_r:=Z\cap\bar B_r(0)\subseteq\R^n$ is compact and since the continuous function $f$ is uniformly continuous on the compact $\Omega_1\tm\Omega\tm Z_r$, there exists a $\delta>0$ such that for all $x\in\Omega_1$, $y\in\Omega$ and $z,\bar z\in Z_r$ one has $\abs{z-\bar z}<\delta\Rightarrow\abs{f(x,y,z)-f(x,y,\bar z)}<\tfrac{\eps}{2\mu(\Omega)}$. Let $u,\bar u\in U_\alpha\cap\bar B_r(0,C_n^0(\Omega))$ with $\norm{u-\bar u}_0<\delta$. Then the inclusions $u(y),\bar u(y)\in Z_r$ and the estimate $\abs{u(y)-\bar u(y)}<\delta$ hold for all $y\in\Omega$. They yield
	$$
		\abs{\sU(u)(x)-\sU(u_0)(x)}
		\stackrel{\eqref{udef}}{\leq}
		\int_\Omega\abs{f(x,y,u(y))-f(x,y,u_0(y))}\d\mu(y)
		\leq
		\tfrac{\eps}{2}\fall x\in\Omega_1
	$$
	and passing to the least upper bound over $x\in\Omega_1$ results in $\norm{\sU(u)-\sU(u_0)}_0<\eps$, i.e.\ $\sU$ is uniformly continuous on each set $U_\alpha\cap\bar B_r(0,C_n^0(\Omega))$. 
\end{proof}

The subsequent assumption allows us to infer H\"older continuity of $\sU$. 
\begin{hypo}
	Let $\vartheta\in(0,1]$. 
	\begin{itemize}
		\item[$(U_0')$] For every $r>0$ there exists a function $l_r:\Omega_1\tm\Omega\to\R_+$, measurable in the second argument and satisfying $\esup_{\xi\in\Omega_1}\int_\Omega l_r(\xi,y)\d\mu(y)<\infty$, so that for almost all $y\in\Omega$ the following holds: 
		\begin{align}
			\abs{f(x,y,z)-f(x,y,\bar z)}&\leq l_r(x,y)\abs{z-\bar z}^\vartheta
			\fall x\in\Omega_1,\,z,\bar z\in Z\cap\bar B_r(0). 
			\label{ua04}
		\end{align}
	\end{itemize}
\end{hypo}
Obviously, the condition $(U_0')$ is sufficient for $(U_0^0)$. 

\begin{cor}\label{corulip}
	If additionally $(U_0')$ holds, then $\sU:U_\alpha\to C_d^0(\Omega_1)$ is H\"older on bounded sets, that is 
	\begin{equation}
		\intcc{\sU|_{U_\alpha\cap\bar B_r(0,C_n^0(\Omega))}}_\vartheta
		\leq
		\esup_{\xi\in\Omega_1}\int_\Omega l_r(\xi,y)\d\mu(y)\fall r>0. 
		\label{corulip1}
	\end{equation}
\end{cor}
\begin{proof}
	Since $(U_0')$ implies $(U_0^0)$, we obtain from \pref{propuwell} that $\sU:U_\alpha\to C_d^0(\Omega_1)$ is well-defined. Given $r>0$, for $u,\bar u\in U_\alpha\cap\bar B_r(0,C_n^0(\Omega))$ we derive from $(U_0')$ that
	\begin{align*}
		\abs{[\sU(u)-\sU(\bar u)](x)}
		&\stackrel{\eqref{udef}}{\leq}
		\int_\Omega\abs{f(x,y,u(y))-f(x,y,\bar u(y))}\d\mu(y)
		\stackrel{\eqref{ua04}}{\leq}
		\int_\Omega l_r(x,y)\d\mu(y)\norm{u-\bar u}_0^\vartheta\\
		&\leq
		\esup_{\xi\in\Omega_1}\int_\Omega l_r(\xi,y)\d\mu(y)\norm{u-\bar u}_0^\vartheta
		\fall x\in\Omega_1
	\end{align*}
	and consequently $\norm{\sU(u)-\sU(\bar u)}_0\leq\esup_{\xi\in\Omega_1}\int_\Omega l_r(\xi,y)\d\mu(y)\norm{u-\bar u}_0^\vartheta$ after passing to the least upper bound over all $x\in\Omega_1$.
\end{proof}

Let us proceed to Urysohn operators having values in H\"older spaces with positive exponent $\beta$ rather than in $C_d^0(\Omega_1)$. This requires to sharpen our above assumptions beyond $(U_0^k)$: 
\begin{hypo}
	Let $m\in\N_0$ and $\beta\in(0,1]$. With $0\leq k\leq m$ one assumes that for every $r>0$ there exists
	\begin{itemize}
		\item[$(\bar U_1^k)$] an integrable function $\bar h_r^k:\Omega\to\R_+$, so that for almost all $y\in\Omega$ the following holds: 
		\begin{align}
			\abs{D_3^kf(x,y,z)-D_3^kf(\bar x,y,z)}&\leq\bar h_r^k(y)d(x,\bar x)^\beta
			\text{ for all }x,\bar x\in\Omega_1,\,z\in Z\cap\bar B_r(0), 
			\label{ub01}
		\end{align}

		\item[$(\bar U_3^k)$] a function $c_r^k:\R_+\tm\Omega\to\R_+$ measurable in the second argument with $\lim_{\delta\searrow 0}\int_\Omega c_r^k(\delta,y)\d\mu(y)=0$, so that for almost all $y\in\Omega$ and all $\delta>0$ the following holds:
		\begin{equation}
			\hspace*{-5mm}
			\abs{z-\bar z}\leq\delta
			\Rightarrow
			\abs{D_3^kf(x,y,z)-D_3^kf(x,y,\bar z)-\intcc{D_3^kf(\bar x,y,z)-D_3^kf(\bar x,y,\bar z)}}
			\leq
			c_r^k(\delta,y)d(x,\bar x)^\beta
			\label{ub03s}
		\end{equation}
		for all $x,\bar x\in\Omega_1$, $z,\bar z\in Z\cap\bar B_r(0)$. 
	\end{itemize}
\end{hypo}
\begin{rem}\label{rem22}
	(1) Note that $(\bar U_1^k)$ implies $(U_1^k)$ with the function $h_r^k(x,x_0,y):=\bar h_r^k(y)d(x,x_0)^\beta$ and in particular the limit relation \eqref{ua00} holds uniformly in $x_0\in\Omega_1$. 

	(2) Since it might be tedious to verify the implication \eqref{ub03s}, we note some sufficient conditions:
	\begin{itemize}
		\item If $Z\subseteq\R^n$ is convex, then $(\bar U_1^{k+1})$ implies $(\bar U_3^k)$ with $c_r^k(\delta,y):=\delta\bar h_r^{k+1}(y)$. Indeed, the Mean Value Theorem \cite[p.~341, Thm.~4.2]{lang:93} yields for almost all $y\in\Omega$ that
		\begin{align*}
			&
			\abs{D_3^kf(x,y,z)-D_3^kf(x,y,\bar z)-\intcc{D_3^kf(\bar x,y,z)-D_3^kf(\bar x,y,\bar z)}}\\
			=&
			\abs{\int_0^1D_3^{k+1}f(x,y,\bar z+\theta(z-\bar z))-D_3^{k+1}f(\bar x,y,\bar z+\theta(z-\bar z))\d\theta\intcc{z-\bar z}}\\
			\stackrel{\eqref{ub01}}{\leq}&
			\,\bar h_r^{k+1}(y)\abs{z-\bar z} d(x,\bar x)^\beta
			\fall x,\bar x\in\Omega_1,\,z,\bar z\in Z\cap\bar B_r(0). 
		\end{align*}

		\item Let $\Omega_1\subset\R^{\nu}$ be bounded and convex. Assume for all $r>0$ there is a function $\gamma_r^k:\R_+\tm\Omega\to\R_+$ measurable in the second argument satisfying $\lim_{\delta\searrow 0}\int_\Omega\gamma_r^k(\delta,y)\d\mu(y)=0$, so that for almost all $y\in\Omega$ and all $\delta>0$ the following holds: $D_1D_3f(\cdot,y,z):\Omega_1\to L_(\R^\nu,L_k(\R^n,\R^d))$ exists and is continuous for all $z\in Z$, almost all $y\in\Omega$, 
		$$
			\abs{z-\bar z}<\delta
			\quad\Rightarrow\quad
			\abs{D_1D_3^kf(x,y,z)-D_1D_3^kf(x,y,\bar z)}
			\leq
			\gamma_r^k(\delta,y)
			\fall x,\bar x\in\Omega_1
		$$
		and $z,\bar z\in Z\cap\bar B_r(0)$. If $\abs{z-\bar z}<\delta$, then the Mean Value Theorem \cite[p.~341, Thm.~4.2]{lang:93} yields
		\begin{align*}
			&
			\abs{D_3^kf(x,y,z)-D_3^kf(x,y,\bar z)-\intcc{D_3^kf(\bar x,y,z)-D_3^kf(\bar x,y,\bar z)}}\\
			=&
			\abs{
			\int_0^1D_1D_3^kf(\bar x+\theta(x-\bar x),y,z)-D_1D_3^kf(\bar x+\theta(x-\bar x),y,\bar z)\d\theta
			}\abs{x-\bar x}\\
			\leq&
			(\diam\Omega_1)^{1-\beta}
			\gamma_r^k(\delta,y)\abs{x-\bar x}^\beta
			\fall x,\bar x\in\Omega_1, 
		\end{align*}
		which allows us to choose $c_r^k(\delta,y):=(\diam\Omega_1)^{1-\beta}\gamma_r^k(\delta,y)$. 
	\end{itemize}
	However, this requires one higher order of continuous partial differentiability for the kernel function $f$. 

	(3) Replacing the function $c_r^k:\R_+\tm\Omega\to\R_+$ in $(\bar U_3^k)$ with
	\begin{align}
		\bar c_r^k:\R_+\tm\Omega&\to\R_+,&
		\bar c_r^k(\delta,y)&:=\sup_{\rho\leq\delta}c_r^k(\rho,y)
		\label{cmod}
	\end{align}
	yields a nondecreasing function $c_r^k(\delta,y)\leq\bar c_r^k(\delta,y)$ inheriting the other relevant properties from $c_r^k$. 
\end{rem}

\begin{thm}[well-definedness of $\sU$]\label{thmuwellb}
	Assume that $(U_0^0,\bar U_1^0,U_2^0)$ hold. Then an Urysohn operator $\sU:U_\alpha\to C_d^\beta(\Omega_1)$ is well-defined and bounded. 
	If additionally $(\bar U_3^0)$ holds, then $\sU$ is continuous. 
\end{thm}
\begin{proof}
	Choose $u\in U_\alpha$ and $r>0$ so large that $\norm{u}_0\leq r$ holds. Referring to \rref{rem22}(1) we can apply \pref{propuwell}, which guarantees that $\sU:U_\alpha\to C_d^0(\Omega_1)$ is well-defined, bounded and continuous. 

	(I) Claim: \emph{$\sU:U_\alpha\to C_d^\beta(\Omega_1)$ is well-defined and bounded.}\\
	Given arbitrary $x,\bar x\in\Omega_1$, using $(\bar U_1^0)$ the estimate
	\begin{displaymath}
		\abs{\sU(u)(x)-\sU(u)(\bar x)}
		\stackrel{\eqref{udef}}{\leq}
		\int_\Omega\abs{f(x,y,u(y))-f(\bar x,y,u(y))}\d\mu(y)
		\stackrel{\eqref{ub01}}{\leq}
		\int_\Omega\bar h_r^0(y)\d\mu(y)d(x,\bar x)^\beta
	\end{displaymath}
	implies $\sU(u)\in C_d^\beta(\Omega_1)$ ($\sU$ is well-defined) and $\sup_{\norm{u}\leq r}[\sU(u)]_\beta<\infty$ ($\sU$ is bounded). 

	(II) Claim: \emph{If $(\bar U_3^0)$ holds, then $\sU_0:U_\alpha\to C_d^\beta(\Omega_1)$ is continuous.}\\
	Let $(u_l)_{l\in\N}$ be a sequence in $U_\alpha$ satisfying $\lim_{l\to\infty}\norm{u_l-u}_0=0$. If $r>0$ is chosen sufficiently large that $u,u_l\in\bar B_r(0,C_n^0(\Omega))$ holds for all $l\in\N$, then $(\bar U_3^0)$ yields
	\begin{align*}
		&
		\abs{[\sU(u_l)-\sU(u)](x)-[\sU(u_l)-\sU(u)](\bar x)}\\
		\stackrel{\eqref{udef}}{\leq}&
		\int_\Omega\abs{f(x,y,u_l(y))-f(x,y,u(y))-\intcc{f(\bar x,y,u_l(y))-f(\bar x,y,u(y))}}\d\mu(y)\\
		\stackrel{\eqref{ub03s}}{\leq}&
		\int_\Omega c_r^0(\abs{u_l(y)-u(y)},y)\d\mu(y)d(x,\bar x)^\beta\\
		\stackrel{\eqref{cmod}}{\leq}&
		\int_\Omega\bar c_r^0(\norm{u_l-u}_0,y)\d\mu(y)d(x,\bar x)^\beta
		\fall x,\bar x\in\Omega_1, 
	\end{align*}
	hence, $\intcc{\sU(u_l)-\sU(u)}_\beta\leq\int_\Omega c_r^0(\norm{u_l-u}_0,y)\d\mu(y)$. This shows $\lim_{l\to\infty}\intcc{\sU(u_l)-\sU(u)}_\beta=0$ and combined with \eqref{ucont0} the claim results. 
\end{proof}

Completely continuity of $\sU$ can be achieved by e.g.\ slightly increasing the image space:
\begin{cor}[complete continuity of $\sU$]
	If additionally $(\bar U_3^0)$ holds, then $\sU:U_\alpha\to C_d^\gamma(\Omega_1)$ is completely continuous, provided one of the following holds:
	\begin{itemize}
		\item[(i)] $\alpha\in(0,1]$, $\gamma=\beta$ and $Z$ is closed, 

		\item[(ii)] $\Omega_1$ is bounded, $\alpha\in(0,1]$, $\gamma\in[0,\beta)$ and $Z$ is closed, 

		\item[(iii)] $\Omega_1$ is compact, $\gamma\in[0,\beta]$, $(U_1^0)$ holds with $\lim_{x\to x_0}\frac{\int_\Omega h_r^0(x,x_0,y)\d\mu(y)}{d(x,x_0)^\beta}=0$ uniformly in $x_0\in\Omega_1$,

		\item[(iv)] $\Omega_1$ is compact and $\gamma\in[0,\beta)$.
	\end{itemize}
\end{cor}
\begin{proof}
	We write $\sU_0:U_\alpha\to C_d^\beta(\Omega_1)$ and $\sU_\alpha^\gamma:U_\alpha\to C_d^\gamma(\Omega_1)$ for the operator given in \eqref{udef}. 

	(I) Claim: \emph{If $\alpha\in(0,1]$, then $\sU_0:U_\alpha\to C_d^\beta(\Omega_1)$ is completely continuous.}\\
	Let $(u_l)_{l\in\N}$ be a bounded sequence in $U_\alpha$, i.e.\ there exists a $r>0$ such that 
	\begin{align}
		\sup_{x\in\Omega}\abs{u_l(x)}&\leq r,&
		\sup_{\substack{x,\bar x\in\Omega,\\ x\neq \bar x}}
		\frac{\abs{u_l(x)-u_l(\bar x)}}{d(x,\bar x)^\alpha}\leq r\fall l\in\N. 
		\label{nopachale}
	\end{align}
	Thus, the subset $\set{u_l}_{l\in\N}\subset C_n^0(\Omega)$ is bounded and equicontinuous. By the Arzel{\`a}-Ascoli theorem \cite[p.~31, Thm.~3.2]{martin:76} there exists a subsequence $(u_{k_l})_{l\in\N}$ and a $u\in C_n^0(\Omega)$ with $\lim_{l\to\infty}\|u_{k_l}-u\|_0=0$. Because $Z$ is closed, we have $u(x)\in Z$ for all $x\in\Omega$, i.e.\ $u\in U_0$. Since \eqref{nopachale} also holds for each $u_{k_l}$, passing to the limit $l\to\infty$ shows $u\in U_\alpha$. The continuity shown in \tref{thmuwellb} implies $\lim_{l\to\infty}\|\sU_0(u_{k_l})-\sU_0(u)\|_\beta=0$. This establishes that every bounded sequence in $\sU_0(U_\alpha)\subset C_d^\beta(\Omega_1)$ has a convergent subsequence, i.e.\ the image $\sU_0\bigl(U_\alpha\cap\bar B_r(0,C_n^\alpha(\Omega))\bigr)$ is relatively compact. Therefore, $\sU_0$ maps bounded subsets of $U_\alpha$ into relatively compact sets. This shows that $\sU_0$ is completely continuous. 

	(II) Claim: \emph{If (iii) holds, then $\sU_0^\beta:U_0\to C_d^\beta(\Omega_1)$ is completely continuous.}\\
	Let $B\subset C_n^0(\Omega)$ be bounded and choose $r>0$ so large that $\norm{u}_0\leq r$ holds for all $u\in B$. We establish that $\sU_0^\beta(B)\subset C_d^\beta(\Omega_1)$ fulfills the assumptions of \tref{thmaa}. First, $\sU_0^\beta(B)$ is bounded due to \tref{thmuwellb}. Second, given $\eps>0$ by assumption (iii) there exists a $\delta>0$ such that $d(x,\bar x)\leq\delta$ implies the estimate
	$
		\frac{\int_\Omega h_r(x,\bar x,y)\d\mu(y)}{d(x,\bar x)^\beta}\leq\eps
	$
	for all $x,\bar x\in\Omega_1$, $x\neq\bar x$. Consequently, we obtain that
	$$
		\abs{\sU(u)(x)-\sU(u)(\bar x)}
		\stackrel{\eqref{udef}}{\leq}
		\int_\Omega\abs{f(x,y,u(y))-f(\bar x,y,u(y))}\d\mu(y)
		\stackrel{\eqref{ua01}}{\leq}
		\int_\Omega h_r^0(x,x_0,y)\d\mu(y)
		\leq
		\eps d(x,\bar x)^\beta. 
	$$
	Thus, the bounded $\sU_0^\beta(B)\subset C_d^\beta(\Omega_1)$ is relatively compact. Hence, $\sU_0^\beta$ is completely continuous.
	
	(III) Under 
	(i) the mapping $\sU=\sU_0$ is completely continuous due to step (I), 
	under (ii) the map $\sU=\sI_\beta^\gamma\circ\sU_0$ is a composition with a continuous embedding $\sI_\beta^\gamma$ (cf.~\tref{thmembed}) with the completely continuous $\sU_0$, under assumption (iii) the operator $\sU=\sI_\beta^\gamma\sU_0^\beta\sU_\alpha^0$ is a composition of bounded embeddings (see \tref{thmembed}) with the due to step (II) completely continuous $\sU_0^\beta$ and finally under (iv) the embedding $\sI_\beta^\gamma$ in $\sU=\sI_\beta^\gamma\sU_0$ is compact thanks to \tref{thm10}. In conclusion, at least one function in the above compositions is completely continuous and the claim results from \cite[pp.~25--26, Thm.~2.1(2)]{precup:02}. 
\end{proof}

For a Lipschitz condition we have to invest continuous differentiability of the kernel function:
\begin{cor}\label{cor25}
	If additionally $(\bar U_1^1)$, $(U_0')$ with $\vartheta=1$ hold on a convex set $Z\subseteq\R^n$, then $\sU:U_\alpha\to C_d^\beta(\Omega_1)$ is Lipschitz on bounded sets, that is 
	\begin{displaymath}
		\intcc{\sU|_{U_\alpha\cap\bar B_r(0,C_n^0(\Omega))}}_1
		\leq
		\max\set{\esup_{\xi\in\Omega_1}\int_\Omega l_r(\xi,y)\d\mu(y),\int_\Omega\bar h_r^1(y)\d\mu(y)}
		\fall r>0.
	\end{displaymath}
\end{cor}
\begin{proof}
	From \rref{rem22}(1) and \cref{corulip} we obtain that $\sU:U_\alpha\to C_d^0(\Omega_1)$ is Lipschitz on bounded sets. Let $r>0$ and choose $u,\bar u\in U_\alpha\cap\bar B_r(0,C_n^0(\Omega))$. Since $Z$ is convex, the inclusion $u(y)+\theta(u_l(y)-u(y))\in Z$ holds for all $y\in\Omega$ and $\theta\in[0,1]$. Hence, \cite[p.~341, Thm.~4.2]{lang:93} applies and implies
	\begin{align*}
		&
		[\sU(u)-\sU(\bar u)](x)-\intcc{\sU(u)-\sU(\bar u)}(\bar x)\\
		\stackrel{\eqref{udef}}{=}&
		\int_\Omega f(x,y,u(y))-f(x,y,\bar u(y))-\intcc{f(\bar x,y,u(y))-f(\bar x,y,\bar u(y))}\d\mu(y)\\
		=&
		\int_\Omega
		\int_0^1
		D_3f\bigl(x,y,\bar u(y)+\theta(u(y)-\bar u(y))\bigr)-D_3f\bigl(\bar x,y,\bar u(y)+\theta(u(y)-\bar u(y))\bigr)\d\theta\intcc{u(y)-\bar u(y)}\d\mu(y)
	\end{align*}
	and consequently $(\bar U_1^1)$ leads to
	\begin{align*}
		&
		\abs{[\sU(u)-\sU(\bar u)](x)-\intcc{\sU(u)-\sU(\bar u)}(\bar x)}\\
		\leq&
		\int_\Omega\!
		\int_0^1\!
		\abs{D_3f\bigl(x,y,\bar u(y)+\theta(u(y)-\bar u(y))\bigr)\!-\!D_3f\bigl(\bar x,y,\bar u(y)+\theta(u(y)-\bar u(y))\bigr)}\d\theta\abs{u(y)-\bar u(y)}\d\mu(y)\\
		\stackrel{\eqref{ub01}}{\leq}&
		\int_\Omega\bar h_r^1(y)\d\mu(y)d(x,\bar x)^\beta\norm{u-\bar u}_0\fall x,\bar x\in\Omega_1. 
	\end{align*}
	This guarantees the estimate $[\sU(u)-\sU(\bar u)]_\beta\leq\int_\Omega\bar h_r^1(y)\d\mu(y)\norm{u-\bar u}_0$ and combined with \eqref{corulip1} it results that $\sU$ is Lipschitz on bounded sets. 
\end{proof}
\subsection{Continuous differentiability}
In the following, we investigate the smoothness of Urysohn operators \eqref{udef}:
\begin{lem}\label{lemuk}
	Assume that $(U_0^k,U_1^k,U_2^k)$ hold for some $k\in\N$. 
	Then $\sU^k:U_\alpha\to L_k(C_n^\alpha(\Omega),C_d^0(\Omega_1))$ given by
	\begin{equation}
		\sU^k(u)v_1\cdots v_k
		:=
		\int_\Omega D_3^kf\bigl(\cdot,y,u(y)\bigr)v_1(y)\cdots v_k(y)\d\mu(y)
		\fall v_1,\ldots, v_k\in C_n^\alpha(\Omega)
		\label{ukdef}
	\end{equation}
	is well-defined and continuous.
\end{lem}
\begin{proof}
	The well-definedness of $\sU^k$ is shown verbatim to the step (I) of the proof for \pref{propuwell} and we hence focus on continuity. Given $u\in U_\alpha$ let $(u_l)_{l\in\N}$ be a sequence in $U_\alpha$ with $\lim_{l\to\infty}\norm{u_l-u}_0=0$. Choose $r>0$ so large that $u,u_l\in B_r(0,C_n^0(\Omega))$ for all $l\in\N$. Using $(U_0^k)$ this leads to 
	\begin{displaymath}
		\lim_{l\to\infty}D_3^kf(x,y,u_l(y))=D_3^kf(x,y,u(y))\fall x\in\Omega_1\text{ and almost all }y\in\Omega.
	\end{displaymath}
	For $\eps>0$ there exists a $\delta>0$ such that subsets $\tilde\Omega\subseteq\Omega$ satisfying $\mu(\tilde\Omega)\leq\delta$ fulfill $\int_{\tilde\Omega}b_r^k(x,y)\d\mu(y)\leq\tfrac{\eps}{4}$ and Egoroff's theorem \cite[p.~87, Prop.~3.1.3]{cohn:80} yields a $\Omega'\subseteq\Omega$ with $\mu(\Omega')\leq\delta$ and an $L_1\in\N$ such that $\abs{D_3^kf(x,y,u_l(y))-D_3^kf(x,y,u(y))}\leq\tfrac{\eps}{2\mu(\Omega)}$ for all $x\in\Omega_1$, $y\in\Omega\setminus\Omega'$, $l\geq L_1$. From this, for any $x\in\Omega_1$ and integers $l\geq L_1$ we arrive at
	\begin{align}
		\abs{[\sU^k(u_l)-\sU^k(u)]v_1\cdots v_k(x)}
		\stackrel{\eqref{ukdef}}{\leq}&
		\int_{\Omega\setminus\Omega'}\abs{D_3^kf(x,y,u_l(y))-D_3^kf(x,y,u(y))}\d\mu(y)
		\notag\\
		&\quad+
		\int_{\Omega'}\abs{D_3^kf(x,y,u_l(y))-D_3^kf(x,y,u(y))}\d\mu(y)
		\label{ukdefs}\\
		\leq&
		\int_{\Omega\setminus\Omega'}\tfrac{\eps}{2\mu(\Omega)}\d\mu(y)+
		2\int_{\Omega'}b_r^k(x,y)\d\mu(y)
		\leq
		\tfrac{\eps}{2}+2\tfrac{\eps}{4}=\eps.
		\notag
	\end{align}
	Passing first to the supremum over $x\in\Omega_1$ therefore implies
	\begin{equation}
		\norm{[\sU^k(u_l)-\sU^k(u)]v_1\cdots v_k}_0
		\leq
		\eps\fall l\geq L_1,\,v_1,\ldots,v_k\in\bar B_1(0,C_n^\alpha(\Omega))
		\label{nolimits}
	\end{equation}
	and second over the vectors $v_1,\ldots,v_k$ yields $\bigl\|\sU^k(u_l)-\sU^k(u)\bigr\|_{L_k(C_n^\alpha(\Omega),C_d^0(\Omega_1))}\leq\eps$ for all $l\geq L_1$. Hence, since $u$ was arbitrary, $\sU^k$ is continuous. 
\end{proof}

\begin{prop}[continuous differentiability of $\sU$]\label{propuder}
	Let $m\in\N$. Assume that $(U_0^k,U_1^k,U_2^k)$ hold for all $0\leq k\leq m$ on a convex set $Z\subseteq\R^n$. 
	Then an Urysohn operator $\sU:U_\alpha\to C_d^0(\Omega_1)$ is $m$-times continuously differentiable with $D^k\sU=\sU^k$ for every $1\leq k\leq m$. 
\end{prop}
\begin{proof}
	(I) Thanks to \lref{lemuk} the mappings $\sU^k:U_\alpha\to L_k(C_n^\alpha(\Omega),C_d^0(\Omega_1))$ are well-defined and continuous for $0\leq k\leq m$. Let $u\in U_\alpha$ and $h\in C_d^\alpha(\Omega)$ such that $u+h\in U_\alpha$. Due to the convexity of $Z$, the inclusion $u(y)+\theta h(y)\in Z$ holds for all $y\in\Omega$ and $\theta\in[0,1]$. Then the remainder functions
	\begin{displaymath}
		r_k(h):=\sup_{\theta\in[0,1]}\norm{\sU^{k+1}(u+\theta h)-\sU^{k+1}(u)}_{L_{k+1}(C_n^\alpha(\Omega),C_d^0(\Omega_1))}
	\end{displaymath}
	satisfy $\lim_{h\to 0}r_k(h)=0$ for all $0\leq k<m$. Now we obtain from \cite[p.~341, Thm.~4.2]{lang:93} that
	\begin{align*}
		&
		[\sU^k(u+h)-\sU^k(u)-\sU^{k+1}(u)h](x)\\
		\stackrel{\eqref{ukdef}}{=}&
		\int_\Omega D_3^kf(x,y,u(y)+h(y))-D_3^kf(x,y,u(y))-D_3^{k+1}f(x,y,u(y))h(y)\d\mu(y)\\
		=&
		\int_\Omega\int_0^1\intcc{D_3^{k+1}f(x,y,u(y)+\theta h(y))-D_3^{k+1}f(x,y,u(y))}h(y)\d\theta\d\mu(y)\\
		\stackrel{\eqref{ukdef}}{=}&
		\int_0^1\intcc{(\sU^{k+1}(u+\theta h)-\sU^k(u))h}(x)\d\theta
	\end{align*}
	by Fubini's theorem \cite[p.~155, Thm.~14.1]{dibenedetto:16}. Consequently, 
	\begin{align*}
		\abs{[\sU^k(u+h)-\sU^k(u)-\sU^{k+1}(u)h](x)}
		&\leq
		\int_0^1\norm{\sU^{k+1}(u+\theta h)-\sU^{k+1}(u)}_{L_{k+1}(C_n^\alpha(\Omega),C_d^0(\Omega_1))}\d\theta\norm{h}_0\\
		&\leq
		r_{k+1}(h)\norm{h}_\alpha\fall x\in\Omega_1
	\end{align*}
	and after passing to the least upper bound over $x\in\Omega_1$ it results
	\begin{equation}
		\norm{\sU^k(u+h)-\sU^k(u)-\sU^{k+1}(u)h}_{L_k(C_n^\alpha(\Omega),C_d^0(\Omega_1))}
		\leq 
		r_k(h)\norm{h}_\alpha.
		\label{nor1}
	\end{equation}
	This establishes that $\sU^k:U_\alpha\to L_k(C_n^\alpha(\Omega),C_d^0(\Omega_1))$ is differentiable in $u$ with the derivative $\sU^{k+1}(u)$. 

	(II) Applying step (I) in case $k=0$ shows that $\sU$ is differentiable on $U_\alpha$ with the derivative $\sU^1$. Given this, mathematical induction yields that $\sU:U_\alpha\to C_d^0(\Omega_1)$ is actually $m$-times differentiable with the derivatives $D^k\sU=\sU^k$ for all $1\leq k\leq m$, which in turn are continuous due to \lref{lemuk}. 
\end{proof}

We proceed to Urysohn operators having values in a H\"older space. 
\begin{lem}\label{lemukb}
	Assume that $(U_0^k,\bar U_1^k,U_2^k)$ hold for some $k\in\N$. Then $\sU^k:U_\alpha\to L_k(C_n^\alpha(\Omega),C_d^\beta(\Omega_1))$ given by \eqref{ukdef} is well-defined. If additionally $(\bar U_3^k)$ holds, then $\sU^k$ is continuous. 
\end{lem}
\begin{proof}
	The well-definedness of $\sU^k$ follows as in step (I) from the proof of \tref{thmuwellb}. Let $u\in U_\alpha$ and $(u_l)_{l\in\N}$ denote a sequence in $U_\alpha$ fulfilling the limit relation $\lim_{l\to\infty}\norm{u_l-u}_0=0$. In addition, choose $r>0$ sufficiently large so that the inclusion $u,u_l\in\bar B_r(0,C_n^0(\Omega))$ for each $l\in\N$ holds. Therefore, 
	\begin{align*}
		&
		[(\sU^k(u_l)-\sU^k(u))v_1\cdots v_k](x)-[(\sU^k(u_l)-\sU^k(u))v_1\cdots v_k](\bar x)\\
		\stackrel{\eqref{ukdef}}{=}&
		\int_\Omega\intcc{D_3^kf(x,y,u_l(y))-D_3^kf(x,y,u(y))-[D_3^kf(\bar x,y,u_l(y))-D_3^kf(\bar x,y,u(y))]}\\
		&\qquad\cdot v_1(y)\cdots v_k(y)\d\mu(y)
	\end{align*}
	and after passing to the norm our assumption $(\bar U_3^k)$ results in
	\begin{align*}
		&
		\abs{[(\sU^k(u_l)-\sU^k(u))v_1\cdots v_k](x)-[(\sU^k(u_l)-\sU^k(u))v_1\cdots v_k](\bar x)}\\
		\stackrel{\eqref{ub03s}}{\leq}&
		\int_\Omega c_r^k(\abs{u_l(y)-u(y)},y)\d\mu(y)d(x,\bar x)^\beta
		\stackrel{\eqref{cmod}}{\leq}
		\int_\Omega\bar c_r^k(\norm{u_l-u}_0,y)\d\mu(y)d(x,\bar x)^\beta
		\fall x,\bar x\in\Omega_1.
	\end{align*}
	Consequently, for $v_1,\ldots,v_k\in\bar B_1(0,C_n^\alpha(\Omega))$ we derive the estimate
	\begin{equation}
		[(\sU^k(u_l)-\sU^k(u))v_1\cdots v_k]_\beta
		\leq
		\int_\Omega\bar c_r^k(\norm{u_l-u}_0,y)\d\mu(y). 
		\label{lemukb3}
	\end{equation}
	In particular, given $\eps>0$ there exists a $L_2\in\N$ such that
	\begin{displaymath}
		[(\sU^k(u_l)-\sU^k(u))v_1\cdots v_k]_\beta
		\leq
		\eps
		\fall l\geq L_2,\,v_1,\ldots,v_k\in\bar B_1(0,C_n^\alpha(\Omega)).
	\end{displaymath}
	In conclusion, with \eqref{nolimits} this implies
	\begin{align*}
		\norm{(\sU^k(u_l)-\sU^k(u))v_1\cdots v_k}_\beta
		&=
		\max\set{\norm{(\sU^k(u_l)-\sU^k(u))v_1\cdots v_k}_0,\intcc{(\sU^k(u_l)-\sU^k(u))v_1\cdots v_k}_\beta}\\
		&\leq
		\eps
		\fall v_1,\ldots,v_k\in\bar B_1(0,C_n^\alpha(\Omega))
	\end{align*}
	and thus $\bigl\|\sU^k(u_l)-\sU^k(u)\bigr\|_{L_k(C_n^\alpha(\Omega),C_d^\beta(\Omega_1))}\leq\eps$ for all $l\geq\max\set{L_1,L_2}$. Therefore, because the function $u$ was arbitrarily chosen, $\sU^k$ is continuous on $U_\alpha$. 
\end{proof}

In contrast to \pref{propuder}, establishing continuous differentiability now requires to invest one additional order of differentiability on the kernel function:
\begin{thm}[continuous differentiability of $\sU$]\label{thmuderb}
	Let $m\in\N$. Assume that $(U_0^k,\bar U_1^k,U_2^k,\bar U_3^k)$ hold for all $0\leq k\leq m$ on a convex set $Z\subseteq\R^n$. Then an Urysohn operator $\sU:U_\alpha\to C_d^\beta(\Omega_1)$ is $m$-times continuously differentiable with $D^k\sU=\sU^k$ for every $1\leq k\leq m$. 
\end{thm}
\begin{proof}
	We establish the assertion for $\sU:U_\alpha\to C_d^\beta(\Omega_1)$ first. 
	Let $u\in U_\alpha$ and $h\in C_n^\alpha(\Omega)$ such that $u+h\in U_\alpha$. Moreover, choose $r>0$ so large that $u,u+h\in\bar B_r(0,C_n^0(\Omega))$ holds. 

	(I) Let $0\leq k<m$. Above all, with the function $\bar c_r^k:\R_+
\tm\Omega_1\to\R_+$ in \eqref{cmod} we observe that 
	$$
		\rho_k(h)
		:=
		\int_0^1\int_\Omega
		\bar c_r^{k+1}(\theta\norm{h}_0,y)\d\mu(y)\d\theta
	$$
	satisfies $\lim_{h\to 0}\rho_k(h)=0$. Given arbitrary $x,\bar x\in\Omega_1$, again \cite[p.~341, Thm.~4.2]{lang:93} yields 
	\begin{align*}
		&
		[\sU^k(u+h)-\sU^k(u)-\sU^{k+1}(u)](x)-
		[\sU^k(u+h)-\sU^k(u)-\sU^{k+1}(u)](\bar x)\\
		\stackrel{\eqref{ukdef}}{=}&
		\int_\Omega\int_0^1
		\bigl[D_3^{k+1}f(x,y,u(y)+\theta h(y))-D_3^{k+1}f(x,y,u(y))\\
		&
		\quad-(D_3^{k+1}f(\bar x,y,u(y)+\theta h(y))-D_3^{k+1}f\bigl(\bar x,y,u(y))\bigr)\bigr]\d\theta h(y)\d\mu(y)\\
		=&
		\int_0^1\int_\Omega
		\bigl[D_3^{k+1}f(x,y,u(y)+\theta h(y))-D_3^{k+1}f(x,y,u(y))\\
		&
		\quad-(D_3^{k+1}f\bigl(\bar x,y,u(y)+\theta h(y))-D_3^{k+1}f(\bar x,y,u(y))\bigr)\bigr]h(y)\d\mu(y)\d\theta
	\end{align*}
	due to Fubini's theorem \cite[p.~155, Thm.~14.1]{dibenedetto:16} and using the assumption $(\bar U_3^{k+1})$ we obtain
	\begin{align*}
		&
		\abs{[\sU^k(u+h)-\sU^k(u)-\sU^{k+1}(u)h](x)-[\sU^k(u+h)-\sU^k(u)-\sU^{k+1}(u)h](\bar x)}\\
		\stackrel{\eqref{ub03s}}{\leq}&
		\int_0^1\int_\Omega
		c_r^{k+1}(\theta\abs{h(y)},y)\abs{h(y)}\d\mu(y)\d\theta d(x,\bar x)^\beta\\
		\stackrel{\eqref{cmod}}{\leq}&
		\int_0^1\int_\Omega
		\bar c_r^{k+1}(\theta\norm{h}_0,y)\d\mu(y)\d\theta d(x,\bar x)^\beta\norm{h}_\alpha,
	\end{align*}
	which in turn implies
	\begin{equation}
		\intcc{\sU^k(u+h)-\sU^k(u)-\sU^{k+1}(u)h}_\beta
		\leq
		\rho_k(h)\norm{h}_\alpha.
		\label{nor2}
	\end{equation}
	If we combine this with the inequalities \eqref{nor1}, then $\|\sU^k(u+h)-\sU^k(u)-\sU^{k+1}(u)h\|_\beta
\leq R_k(h)\norm{h}_\alpha$ with the remainder term $R_k(h):=\max\set{r_k(h),\rho_k(h)}$ satisfying the desired limit relation $\lim_{h\to 0}R_k(h)=0$. 
	
	(II) Applying step (I) in case $k=0$ shows that $\sU$ is differentiable on $U_\alpha$ with the derivative $\sU^1$. Given this, mathematical induction yields that $\sU:U_\alpha\to C_d^\beta(\Omega_1)$ is actually $m$-times differentiable with the derivatives $D^k\sU=\sU^k$ for all $0\leq k\leq m$. Their continuity is guaranteed by \lref{lemukb}. 
\end{proof}

We close our general analysis of Urysohn operators with several remarks: 
\begin{rem}[boundedness and continuity of $\sU$]
	(1) The boundedness of Urysohn operators $\sU$ stated in \pref{propuwell} and \tref{thmuwellb} actually means that the $\sU$-images of merely $\norm{\cdot}_0$-bounded subsets $B\subset U_\alpha$ are bounded. This means that the functions in $B$ need not to have uniformly bounded H\"older constants. 

	(2) The continuity statements for the Urysohn operator $\sU$ in \pref{propuwell} and \tref{thmuwellb}, as well as for its derivatives $D^k\sU$ in \pref{propuder} and \tref{thmuderb} are to be understood in the following strong form: Already convergence in the domain $U_\alpha$ w.r.t.\ the norm $\norm{\cdot}_0$ is sufficient for convergence of the $\sU$-values in the norm $\norm{\cdot}_0$ resp.\ $\norm{\cdot}_\beta$. A corresponding statement applies to both \cref{corulip} and \ref{cor25}. 
\end{rem}

\begin{rem}[Urysohn operators $\sU:U_\alpha\to C_d^\gamma(\Omega_1)$]
	The above statements extend to Urysohn operators mapping into the $\gamma$-H\"older functions over bounded metric spaces $\Omega_1$. This is due to the corresponding representation $\sI_\beta^\gamma\sU:U_\alpha\to C_d^\gamma(\Omega_1)$, where the embedding $\sI_\beta^\gamma$ from \tref{thmembed} is continuous.
\end{rem}

\begin{rem}[Nystr\"om methods]
	Let $\Omega^{(l)}$ be a discrete subset of a compact set $\Omega_1\subset\R^\kappa$ and suppose $w_\eta\geq 0$ are nonnegative reals, $\eta\in\Omega^{(l)}$. In \eref{exleb} resp.\ \ref{exnyst} we pointed out that both Urysohn operators
	\begin{align*}
		\sU:U_\alpha&\to C_d^\beta(\Omega_1),&
		\sU_l:U_\alpha^{(l)}&\to C_d^\beta(\Omega_1),\\
		\sU(u)&:=\int_{\Omega_1}f(\cdot,y,u(y))\d y,&
		\sU_l(u)&:=\sum_{\eta\in\Omega^{(l)}}w_\eta f(\cdot,\eta,u(\eta))
	\end{align*}
	fit well into our abstract setting, where we abbreviated $U_\alpha^{(l)}:=\set{u:\Omega^{(l)}\to Z\mid u\in C_n^\alpha(\Omega^{(l)})}$. However, when dealing with Nystr\"om methods or for iterating integral operators $\sU_l$ it is desirable to work with Urysohn operators defined on $U_\alpha$ rather than $U_\alpha^{(l)}$. For this purpose, let us introduce the linear operator $E_l:C_n^\alpha(\Omega_1)\to C_n^\alpha(\Omega^{(l)})$ given by $E_lu:=u|_{\Omega^{(l)}}$. It satisfies $E_lU_\alpha\subseteq U_\alpha^{(l)}$ and is easily seen to be bounded with $\norm{E_l}_{C_n^\alpha(\Omega_1),C_n^\alpha(\Omega^{(l)})}\leq 1$. Hence, rather than $\sU_l$ we consider the composition
	\begin{align*}
		\sU_l':U_\alpha&\to C_d^\beta(\Omega_1),&
		\sU_l'(u):=\sU_l(E_lu)=\sum_{\eta\in\Omega^{(l)}}w_\eta f(\cdot,\eta,u(\eta)),
	\end{align*}
	which, under appropriate assumptions on the kernel function $f$, inherits its properties from $\sU_l$. 
\end{rem}
\subsection{Convolutive operators}
\label{sec23}
In our above analysis the H\"older continuity of an image $\sU(u):\Omega_1\to\R^d$ of a general Urysohn operator \eqref{udef} was guaranteed and prescribed by the exponent of the kernel function $f$ in its first variable from assumption $(\bar U_1^0)$. A higher degree of smoothness cannot be expected, as simple examples like the kernel function $f(x,y,z):=f_1(x)$ illustrate, where $\sU(u)(x)=\int_\Omega f_1(x)\d\mu(y)=\mu(\Omega)f_1(x)$ inherits its smoothness from $f_1$. This situation changes for kernel functions of convolution type. Here the smoothness (H\"older continuity, differentiability) of the arguments $u$ transfers to the images $\sU(u)$, i.e.\ such integral operators possess a smoothing property we are about to analyze over the course of this section. Our results generalize those of \cite[pp.~52ff, Sect.~3.4.2]{hackbusch:95} obtained for linear operators. 

To be more precise, let us restrict to a compact interval $\Omega=\Omega_1=[a,b]$, $a<b$, equipped with the $1$-dimensional Lebesgue measure $\mu=\lambda_1$ in \eqref{udef}. Moreover, the kernel function is of the form
$$
	f(x,y,z)=\tilde f(x-y,z)
$$
with a function $\tilde f:[a-b,b-a]\tm Z\to\R^d$. This yields a \emph{convolutive Urysohn operator}
\begin{align}
	\tilde\sU:U_\alpha&\to F([a,b],\R^d),&
	\tilde\sU(u)(x):=\int_a^b\tilde f(x-y,u(y))\d y\fall x\in[a,b]. 
	\label{deftU}
\end{align}
\begin{hypo}
	Let $m\in\N_0$ and $\tilde\Omega:=[a-b,b-a]$. With $0\leq k\leq m$ one assumes: 
	\begin{itemize}
		\item[$(C_0^k)$] The partial derivative $D_2^k\tilde f(y,\cdot):Z\to\R^d$ exists and is continuous for Lebesgue-almost all $y\in\tilde\Omega$, 

		\item[$(C_1^k)$] $D_2^k\tilde f(\cdot,z):\tilde\Omega\to\R^d$ is measurable for all $z\in Z$ and for every $r>0$ there exists an integrable function $\tilde b_r^k:\tilde\Omega\to\R_+$ so that for Lebesgue-almost all $y\in\tilde\Omega$ the following holds: 
		\begin{equation}
			\abs{D_2^k\tilde f(y,z)}\leq\tilde b_r^k(y)\fall z\in Z\cap\bar B_r(0), 
			\label{no16}
		\end{equation}

		\item[$(C_2)$] for all $r>0$ there exists an integrable function $\tilde l_r:\tilde\Omega\to\R_+$, so that for Lebesgue-almost all $y\in\tilde\Omega$ the following holds: 
		\begin{equation}
			\abs{\tilde f(y,z)-\tilde f(y,\bar z)}\leq \tilde l_r(y)\abs{z-\bar z}\fall z,\bar z\in Z\cap\bar B_r(0).
			\label{tildelip}
		\end{equation}
	\end{itemize}
\end{hypo}
Note that the Lipschitz condition $(C_2)$ implies $(C_0^0)$, but also \eqref{no16} for $k=1$ with $\tilde b_r^1=\tilde l_r$. 

Under these assumptions the H\"older continuity of $u\in U_\alpha$ carries over to the values $\tilde\sU(u)$: 
\begin{thm}[H\"older continuity of $\tilde\sU(u)$]\label{thmconv1}
	Assume that $(C_1^0,C_2)$ hold and $\alpha\in(0,1]$. 
	If $u\in U_\alpha$, $r>\norm{u}_0$ and there exists a real $C\geq 0$ satisfying 
	\begin{align}
		\int_x^{\bar x}\tilde b_r^0(y)\d y&\leq C(\bar x-x)^\alpha\fall a-b\leq x\leq\bar x\leq b-a,
		\label{coass}
	\end{align}
	then the image satisfies $\tilde\sU(u)\in C_d^\alpha[a,b]$. 
\end{thm}
\begin{proof}
	Let $u\in U_\alpha$ and $r>\norm{u}_0$. 
	
	(I) Let $x\in[a,b]$ be given. Above all, $(C_2)$ implies that $\tilde f(y,\cdot):Z\to\R^d$ is continuous for Lebesgue-almost all $y\in\tilde\Omega$ (i.e.\ the assumption $(C_0^0)$ holds). Combined with the measurability assumed in $(C_1^0)$ we conclude from \cite[p.~62, Lemma~5.1]{precup:02} that $y\mapsto\tilde f(x-y,u(y))$ is measurable. Moreover, due to \eqref{no16} one has the estimate $\abs{\tilde f(x-y,u(y))}\leq\tilde b_r^0(x-y)$, where
	$
		\int_a^b\tilde b_r^0(x-y)\d y
		\leq
		\int_{\tilde\Omega}\tilde b_r^0(\eta)\d\eta
	$
	for all $x\in[a,b]$. Consequently the function $y\mapsto\tilde f(x-y,u(y))$ is integrable. Hence $\tilde\sU(u):[a,b]\to\R^d$ is well-defined. 

	(II) For $x,\bar x\in[a,b]$ with $\bar x=x+\Delta>x$ it results
	\begin{align*}
		\tilde\sU(u)(\bar x)
		&\stackrel{\eqref{deftU}}{=}
		\int_a^b\tilde f(\bar x-y,u(y))\d y
		=
		\int_a^b \tilde f(x+\Delta-y,u(y))\d y\\
		&=
		\int_{a-\Delta}^a\tilde f(x-\eta,u(\eta+\Delta))\d\eta+
		\int_a^{b-\Delta}\tilde f(x-\eta,u(\eta+\Delta))\d\eta
	\end{align*}
	via the substitution $\eta:=y-\Delta$ and analogously 
	$$
		\tilde\sU(u)(x)
		\stackrel{\eqref{deftU}}{=}
		\int_a^{b-\Delta}\tilde f(x-\eta,u(\eta))\d\eta+
		\int_{b-\Delta}^b\tilde f(x-\eta,u(\eta))\d\eta.
	$$
	Whence, the difference $\tilde\sU(u)(\bar x)-\tilde\sU(u)(x)=I_0+I_1+I_2$ can be written as sum of the terms
	\begin{align*}
		I_0&:=\int_a^{b-\Delta}\tilde f(x-\eta,u(\eta+\Delta))-\tilde f(x-\eta,u(\eta))\d\eta,&
		I_1&:=\int_{a-\Delta}^a\tilde f(\eta-\Delta,u(\eta+\Delta))\d\eta,\\
		I_2&:=-\int_{b-\Delta}^b\tilde f(x-\eta,u(\eta))\d\eta
		=
		\int_{x-b}^{x-b+\Delta}\tilde f(y,u(x-y))\d y,
	\end{align*}
	which can be estimated separately as
	\begin{align*}
		\abs{I_0}
		&\leq
		\int_a^{b-\Delta}\abs{\tilde f(x-\eta,u(\eta+\Delta))-\tilde f(x-\eta,u(\eta))}\d\eta
		\stackrel{\eqref{tildelip}}{\leq}
		\int_a^{b-\Delta}\tilde l_r(x-\eta)\abs{u(\eta+\Delta)-u(\eta)}\d\eta\\
		&\leq
		[u]_\alpha\int_a^b\tilde l_r(x-\eta)\d\eta \Delta^\alpha
		\leq
		[u]_\alpha\int_{\tilde\Omega}\tilde l_r(\eta)\d\eta \Delta^\alpha
	\end{align*}
	and \eqref{coass} imply that
	\begin{align*}
		\abs{I_1}
		&\leq
		\int_{a-\Delta}^a\tilde b_r^0(x-\eta)\d\eta
		\leq
		C\Delta^\alpha,&
		\abs{I_2}
		&\leq
		\int_{x-b}^{x-b+\Delta}\tilde b_r^0(y)\d\ y
		\leq
		C\Delta^\alpha.
	\end{align*}
	Hence, with $I_0,I_1$ and $I_2$ also their sum is $\alpha$-H\"older due to \tref{thmsum} and thus $\tilde\sU(u)\in C^\alpha[a,b]$. 
\end{proof}

\begin{thm}[H\"older continuity of $\tilde\sU(u)$]
	Assume that $(C_1^0,C_2)$ hold and $\alpha\in(0,1]$, the kernel function $\tilde f$ and the partial derivative $D_1\tilde f$ exist as a continuous functions on both sets $[a-b,0)\tm Z$ and $(0,b-a]\tm Z$. 
	If $u\in U_\alpha$, $r>\norm{u}_0$ and there exists a $h_0>0$ such that
	$
		\int_{-h_0}^{h_0}\tilde l_r(y)\d y<\infty,
	$
	then $\tilde\sU(u)$ is $\alpha$-H\"older on every subinterval compact in $[a,b]$. 
\end{thm}
\begin{proof}
	Let $u\in U_\alpha$. As in step (I) of the proof to \tref{thmconv1} one shows that $\tilde\sU(u)$ is well-defined. Now suppose that $I\subseteq(a,b)$ is a compact subinterval. For each $x\in I$ we choose $h\in(0,h_0]$ so small that $a\leq x-h$ and $x+h\leq b$ holds. This allows us to represent
	\begin{equation}
		\tilde\sU(u)(x)
		=
		\int_a^b\tilde f(x-y,u(y))\d y
		=
		I_1(x)+I_2(x)+I_2(x)
		\label{usum}
	\end{equation}
	with the functions $I_0,I_1,I_2:I\to\R^d$ given by
	\begin{align*}
		I_0(x)&:=\int_{x-h}^{x+h}\tilde f(x-y,u(y))\d y=\int_{-h}^{h}\tilde f(\eta,u(x-\eta))\d\eta,&
		I_1(x)&:=\int_a^{x-h}\tilde f(x-y,u(y))\d y,\\
		&&
		I_2(x)&:=\int_{x+h}^b\tilde f(x-y,u(y))\d y,
	\end{align*}
	where we applied the substitution $\eta=x-y$ in order to rewrite $I_0(x)$. First, we investigate the parameter integral $I_0$. Thereto, using $(C_2)$ we obtain
	\begin{align*}
		\abs{I_0(x)-I_0(\bar x)}
		&\leq
		\int_{-h}^h\abs{\tilde f(\eta,u(x-\eta))-\tilde f(\eta,u(\bar x-\eta))}\d y
		\stackrel{\eqref{tildelip}}{\leq}
		\int_{-h}^h\tilde l_r(y)\abs{u(x-\eta)-u(\bar x-\eta)}\d y\\
		&\leq
		[u]_\alpha\int_{-h}^h\tilde l_r(y)\d y\abs{x-\bar x}^\alpha\fall x,\bar x\in I
	\end{align*}
	and consequently also $I_0$ is $\alpha$-H\"older on the compact subinterval $I$. Second, the integration variable $y$ in the parameter integrals $I_1(x)$ and $I_2(x)$ satisfies $0<h\leq x-y$ resp.\ $x-y\leq-h<0$. Thus, the functions $I_1,I_2$ are differentiable with the derivatives
	\begin{align*}
		I_1'(x)&=\int_a^{x-h}D_1\tilde f(x-y,u(y))\d y+\tilde f(h,u(x-h)),&
		I_2'(x)&=\int_{x+h}^bD_1\tilde f(x-y,u(y))\d y+\tilde f(-h,u(x+h))
	\end{align*}
	for all $x\in I$. Because these derivatives are bounded on the interval $I$, we conclude from \eref{exA2} that $I_1,I_2$ are $\alpha$-H\"older on $I$. In conclusion, \tref{thmsum} yields that the sum \eqref{usum} is $\alpha$-H\"older. 
\end{proof}

\begin{thm}[continuous differentiability of $\tilde\sU(u)$]
	Assume that $(C_0^1,C_1^1)$ hold on a convex set $Z\subseteq\R^n$, $\tilde f:\tilde\Omega\tm Z\to\R^d$ and $D_2\tilde f(\cdot,z)$ are continuous for all $z\in Z$ and $\alpha\in(0,1]$. If $u:[a,b]\to Z$ is continuously differentiable, then also the image $\tilde\sU(u):[a,b]\to\R^d$ is continuously differentiable with the derivative
	\begin{equation}
		\tilde\sU(u)'(x)
		=
		\tilde f(x-a,u(a))-\tilde f(x-b,u(b))-\int_a^bD_2\tilde f(x-y,u(y))u'(y)\d y
		\fall x\in[a,b].
		\label{uuder}
	\end{equation}
\end{thm}
Since the derivative $\tilde\sU(u)'$ is bounded as a continuous function over the compact interval $[a,b]$, it results from \eref{exA2} that $\tilde\sU(u)$ is $\alpha$-H\"older, $\alpha\in(0,1]$. 
\begin{proof}
	Let $u:[a,b]\to Z$ be continuously differentiable and choose $r>\norm{u}_0$. Since $\tilde f$ is continuous, the assumptions $(C_0^0,C_1^0)$ are satisfied and consequently as shown in step (I) of the proof to \tref{thmconv1} the expression $\tilde\sU(u)$ is well-defined. Now fix some $x\in[a,b]$. 

	(I) Claim: \emph{One has the limit relation
	\begin{equation}
		\lim_{h\to 0}
		\int_a^b\frac{1}{h}\intoo{\tilde f(x-y,u(y+h))-\tilde f(x-y,u(y))}\d y
		=
		\int_a^bD_2\tilde f(x-y,u(y))u'(y)\d y. 
		\label{limconv}
	\end{equation}}
	We define the function
	$$
		F(h,y)
		:=
		\begin{cases}
			\tfrac{1}{h}\intoo{\tilde f(x-y,u(y+h))-\tilde f(x-y,u(y))},&h\neq 0,\\
			D_2\tilde f(x-y,u(y))u'(y),&h=0
		\end{cases}
	$$
	having the following properties: First, $F(h,\cdot):[a,b]\to\R^d$ is integrable for every fixed $h$. In order to see this, let us distinguish two cases:\\
	$h\neq 0$: The continuity of $\tilde f$ and $u$ yield that $F(h,\cdot)$ is integrable.\\
	$h=0$: Due to $(C_1^1)$ the derivative $D_2\tilde f(\cdot,z)$ is measurable for all $z\in Z$ and the continuity of $u$ and $u'$ guarantee that $F(0,\cdot)$ is measurable due to \cite[p.~62, Lemma~5.1]{precup:02}. Moreover, for Lebesgue-almost all $y\in[a,b]$ one has $\abs{F(0,y)}=\abs{D_2\tilde f(x-y,u(y))u'(y)}\leq\tilde b_r^1(x-y)\norm{u'}_0$ from \eqref{no16} and due to 
	\begin{equation}
		\int_a^b\tilde b_r^1(x-y)\d y\leq\int_{\tilde\Omega} b_r^1(\eta)\d\eta
		\label{best}
	\end{equation}
	also the function $F(0,\cdot)$ is integrable.\\
	Second, $F(\cdot,y)$ is continuous in $0$, which readily results from the chain rule \cite[p.~337]{lang:93}. Third, applying the Mean Value Theorem \cite[p.~341, Thm.~4.2]{lang:93} twice leads to
	\begin{align*}
		F(h,y)
		&=
		\int_0^1D_2\tilde f\bigl(x-y,u(y)+\theta(u(y+h)-u(y))\bigr)\d\theta\int_0^1u'(y+\theta h)\d\theta
	\end{align*}
	and hence
	$
		\abs{F(h,y)}
		\stackrel{\eqref{no16}}{\leq}
		\int_0^1\tilde b_r^1(x-y)\d\theta\norm{u'}_0
		\leq
		\tilde b_r^1(x-y)\norm{u'}_0
	$
	for all $h$. Thanks to \eqref{best} the right-hand side of this inequality is bounded above by an integrable function independent of $h$ (cf.~$(C_1^1)$). Combining these three aspects, it is a consequence of the dominated convergence theorem \cite[p.~149, Thm.~10.1]{dibenedetto:16} that taking the limit and integration in \eqref{limconv} can be exchanged, which yields the claim. 

	(II) Using the substitution $\eta=x-y$ we obtain the representation
	$$
		\tilde\sU(u)(x)
		\stackrel{\eqref{deftU}}{=}
		\int_a^b\tilde f(x-y,u(y))\d y
		=
		\int_{x-b}^{x-a}\tilde f(\eta,u(x-\eta))\d\eta, 
	$$
	which in turn yields
	\begin{align*}
		&\tilde\sU(u)(x+h)-\tilde\sU(u)(x)
		=
		\int_{x+h-b}^{x+h-a}\tilde f(\eta,u(x+h-\eta))\d\eta
		-
		\int_{x-b}^{x-a}\tilde f(\eta,u(x-\eta))\d\eta\\
		=&
		-\int_{x-b}^{x+h-b}\tilde f(\eta,u(x+h-\eta))\d\eta
		+\int_{x-b}^{x-a}\tilde f(\eta,u(x+h-\eta))-\tilde f(\eta,u(x-\eta))\d\eta
		+\int_{x-a}^{x+h-a}\tilde f(\eta,u(x+h-\eta))\d\eta\\
		=&
		-\int_{x-b}^{x+h-b}\tilde f(\eta,u(x+h-\eta))\d\eta
		-\int_a^b\tilde f(x-y,u(y+h))-\tilde f(x-y,u(y))\d y
		+\int_{x-a}^{x+h-a}\tilde f(\eta,u(x+h-\eta))\d\eta		
	\end{align*}
	for all $h\in[a-x,b-x]$; note that we re-substituted $\eta=x-y$ in the center term of the above sum. The Mean Value Theorem from the integral calculus applies to each continuous component function $\tilde f_i$, $1\leq i\leq d$, in the first and third term in the above sum. Hence, there exist reals 
	\begin{align*}
		\xi_1^i(h)&\in\intoo{\min\set{x-b,x+h-b},\max\set{x-b,x+h-b}},\\
		\xi_2^i(h)&\in\intoo{\min\set{x-a,x+h-a},\max\set{x-a,x+h-a}},
	\end{align*}
	such that the identities
	\begin{align*}
		-\int_{x-b}^{x+h-b}\tilde f_i(\eta,u(x+h-\eta))\d\eta
		&=
		-h\tilde f_i\bigl(\xi_1^i(h),u(x+h-\xi_1^i(h))\bigr),\\
		\int_{x-a}^{x+h-a}\tilde f_i(\eta,u(x+h-\eta))\d\eta
		&=
		h\tilde f_i\bigl(\xi_2^i(h),u(x+h-\xi_2^i(h))\bigr)
	\end{align*}
	hold, which allow us to conclude
		\begin{align*}
		\tfrac{1}{h}\intoo{\tilde\sU(u)_i(x+h)-\tilde\sU(u)_i(x)}
		=&
		\frac{1}{h}\left(
		\int_{x-a}^{x+h-a}\tilde f(\eta,u(x+h-\eta))\d\eta
		-\int_{x-b}^{x+h-b}\tilde f_i(\eta,u(x+h-\eta))\d\eta
		\right.\\
		&
		\left.
		-\int_a^b\tilde f_i(x-y,u(y+h))-\tilde f_i(x-y,u(y))\d y
		\right)\\
		= &
		\tilde f_i\bigl(\xi_2^i(h),u(x+h-\xi_2^i(h))\bigr)
		-
		\tilde f_i\bigl(\xi_1^i(h),u(x+h-\xi_1^i(h))\bigr)\\
		&
		-
		\int_a^b\frac{1}{h}\intoo{\tilde f_i(x-y,u(y+h))-\tilde f_i(x-y,u(y))}\d y
		\fall 1\leq i\leq d. 
	\end{align*}
	Thanks to the limit relations $\lim_{h\to 0}\xi_1^i(h)=x-b$, $\lim_{h\to 0}\xi_2^i(h)=x-a$ for all $1\leq i\leq d$ and returning to vector notation we consequently arrive at
	\begin{eqnarray*}
		&&
		\lim_{h\to 0}\tfrac{1}{h}\intoo{\tilde\sU(u)(x+h)-\tilde\sU(u)(x)}\\
		& = &
		\tilde f(x-a,u(a))-\tilde f(x-b,u(b))
		-
		\lim_{h\to 0}
		\int_a^b\frac{1}{h}\intoo{\tilde f(x-y,u(y+h))-\tilde f(x-y,u(y))}\d y\\
		& \stackrel{\eqref{limconv}}{=} &
		\tilde f(x-a,u(a))-\tilde f(x-b,u(b))
		-
		\int_a^bD_2\tilde f(x-y,u(y))u'(y)\d y. 
	\end{eqnarray*}
	This establishes that the image $\tilde\sU(u)$ is differentiable in $x\in[a,b]$ with the derivative \eqref{uuder}. Hence, in order to show that $\sU(u)'$ is continuous, it suffices to establish the continuity of the parameter integral $x\mapsto\int_a^b\tilde F(x,y)\d y$ with $\tilde F(x,y):=D_2\tilde f(x-y,u(y))u'(y)$. By assumption follows that $\tilde F(\cdot,y)$ is continuous on $[a,b]$. From \cite[p.~62, Lemma~5.1]{precup:02} we conclude that $F(x,\cdot)$ is measurable and because of
	$$
		\abs{\tilde F(x,y)}=\abs{D_2\tilde f(x-y,u(y))u'(y)}
		\stackrel{\eqref{no16}}{\leq}
		\tilde b_r^1(x-y)\norm{u'}_0\fall x\in[a,b]
	$$
	integrable (uniformly in $x$, cf.~\eqref{best}). Then the dominated convergence theorem \cite[p.~149, Thm.~10.1]{dibenedetto:16} shows that $x\mapsto\int_a^b\tilde F(x,y)\d y$ is continuous, and thus $\tilde\sU(u)$ is continuously differentiable. 
\end{proof}
\section{Hammerstein integral operators}
\label{sec3}
Let $\alpha\in[0,1]$. \emph{Hammerstein operators} are of the form
\begin{align}
	\sH:U_\alpha&\to F(\Omega_1,\R^d),&
	\sH(u)&:=\int_\Omega k(\cdot,y)g(y,u(y))\d\mu(y)
	\label{hdef}
\end{align}
and represent a relevant special case of the Urysohn operators studied in \sref{sec2} having the kernel function $f(x,y,z):=k(x,y)g(y,z)$. Nevertheless, we investigate them as composition of Fredholm and Nemytskii operators. For this reason, let us study these operator classes independently first. 
\subsection{Fredholm integral operators}
A \emph{Fredholm operator} is a linear integral operator of the form
\begin{align}
	\sK:C_p^\alpha(\Omega)&\to F(\Omega_1,\R^d),&
	\sK u&:=\int_\Omega k(\cdot,y)u(y)\d\mu(y)
	\label{kdef}
\end{align}
determined by a matrix-valued \emph{kernel} $k:\Omega_1\tm\Omega\to\R^{d\tm p}$. 

Fredholm operators apparently fit in the set-up of \sref{sec2} with kernel function $f(x,y,z)=k(x,y)z$. Nevertheless, we take the opportunity for formulate our assumptions in terms of integrals over the kernels, rather than over the kernels. Then the corresponding counterparts to $(U_1^0)$ and $(U_2^0)$ read as: 
\begin{hypo}
	\begin{itemize}
		\item[$(K_1)$] $\lim_{x\to x_0}\int_\Omega\abs{k(x,y)-k(x_0,y)}\d\mu(y)=0$ for all $x_0\in\Omega_1$, 

		\item[$(K_2)$] $k(x,\cdot)$ is measurable for all $x\in\Omega_1$ and
		$
			\sup_{\xi\in\Omega_1}\int_\Omega\abs{k(\xi,y)}\d\mu(y)<\infty.
		$
	\end{itemize}
\end{hypo}
\begin{prop}[well-definedness of $\sK$]\label{propfreda}
	Assume that $(K_1,K_2)$ hold. Then a Fredholm operator $\sK$ satisfies $\sK\in L(C_p^\alpha(\Omega),C_d^0(\Omega_1))$ and
	$$
		\norm{\sK}_{L(C_p^\alpha(\Omega),C_d^0(\Omega_1))}
		\leq
		\max\set{1,(\diam\Omega)^\alpha}\sup_{\xi\in\Omega_1}\int_\Omega\abs{k(\xi,y)}\d\mu(y).
	$$
\end{prop}
\begin{proof}
	Let us abbreviate $M:=\sup_{\xi\in\Omega_1}\int_\Omega\abs{k(\xi,y)}\d\mu(y)$ and write $\sK_0:C_p^0(\Omega)\to C_d^0(\Omega_1)$ instead of $\sK$. The inclusion $\sK_0\in L(C_n^0(\Omega),C_d^0(\Omega_1))$ with $\norm{\sK_0}\leq M$ is shown in \cite[p.~244, Satz~1]{fenyo:stolle:82}. In case $\alpha\in(0,1]$ we consider the composition $\sK=\sK_0\sI_\alpha^0$, where the embedding operator $\sI_\alpha^0$ from \eqref{noe} satisfies the estimate $\norm{\sI_\alpha^0}\leq\max\set{1,(\diam\Omega)^\alpha}$ by \tref{thmembed}. This implies the remaining assertions. 
\end{proof}

\begin{cor}[compactness of $\sK$]\label{corfreda}
	A Fredholm operator $\sK\in L(C_p^\alpha(\Omega),C_d^0(\Omega_1))$ is compact, provided one of the following holds:
	\begin{itemize}
		\item[(i)] $\alpha\in(0,1]$, 

		\item[(ii)] $\Omega_1$ is compact, $\alpha=0$ and $(K_1)$ holds uniformly in $x_0\in\Omega_1$. 
	\end{itemize}
\end{cor}
\begin{proof}
	For compactness of $\sK_0\in L(C_n^0(\Omega),C_d^0(\Omega_1))$ we refer to \cite[p.~247, Satz~4]{fenyo:stolle:82}. In case $\alpha\in(0,1]$ we consider the composition $\sK=\sK_0\sI_\alpha^0$ of the continuous $\sK_0$ and the embedding operator $\sI_\alpha^0$ introduced in \eqref{noe}, which is compact due to \tref{thm10}. 
\end{proof}
\begin{rem}
	If $\Omega_1$ is compact and $k:\Omega_1\tm\Omega\to\R^{d\tm p}$ is continuous, then $(K_1,K_2)$ are fulfilled. Hence, \pref{propfreda} and \cref{corfreda} guarantee that $\sK\in L(C_p^\alpha(\Omega),C_d^0(\Omega_1))$ is compact. 
\end{rem}

In order to handle Fredholm operators which map into the H\"older continuous functions a refinement of assumption $(\bar U_1^0)$ is due:
\begin{hypo}
	Let $\beta\in(0,1]$. 
	\begin{itemize}
		\item[$(\bar K_1)$] There exists a continuous function $\tilde h:\Omega_1^2\to\R_+$ such that 
		\begin{equation}
			\int_\Omega\abs{k(x,y)-k(x_0,y)}\d\mu(y)\leq\tilde h(x,x_0)d(x,\bar x)^\beta\fall x,x_0\in\Omega_1.
			\label{khyp2}
		\end{equation}
	\end{itemize}
\end{hypo}
Obviously, $(\bar K_1)$ implies $(K_1)$. 
\begin{thm}[well-definedness of $\sK$]\label{thmkcom}
	Assume that $(\bar K_1,K_2)$ hold. Then a Fredholm operator $\sK$ satisfies $\sK\in L(C_p^\alpha(\Omega),C_d^\beta(\Omega_1))$ and
	$$
		\norm{\sK}_{L(C_p^\alpha(\Omega),C_d^\beta(\Omega_1))}
		\leq
		\max\set{\max\set{1,(\diam\Omega)^\alpha}\sup_{\xi\in\Omega_1}\int_\Omega\abs{k(\xi,y)}\d\mu(y),\sup_{x,x_0\in\Omega_1}\tilde h(x,x_0)}. 
	$$
\end{thm}
\begin{proof}
	We abbreviate $N:=\sup_{x,x_0\in\Omega_1}\tilde h(x,x_0)$. First, 
	$
		\norm{\sK u}_0\leq M\max\set{1,(\diam\Omega)^\alpha}\norm{u}_\alpha
	$
	holds due to \pref{propfreda} ($M\geq 0$ is defined in its proof). Second, the inequality
	\begin{align*}
		\abs{\sK u(x)-\sK u(\bar x)}
		&\stackrel{\eqref{kdef}}{\leq}
		\int_\Omega\abs{k(x,y)-k(\bar x,y)}\abs{u(y)}\d\mu(y)
		\stackrel{\eqref{khyp2}}{\leq}
		Nd(x,\bar x)^\beta\norm{u}_\alpha
		\fall x,\bar x\in\Omega_1
	\end{align*}
	consequently implies that
	$
		[\sK u]_\beta\leq N\norm{u}_\alpha
	$
	holds. A combination of these two estimates finally guarantees that
	$
		\norm{\sK u}_\beta
		=
		\max\set{\norm{\sK u}_0,[\sK u]_\beta}
		\leq
		\max\set{M\max\set{1,(\diam\Omega)^\alpha},N}\norm{u}_\alpha
	$
	and thus
	$$
		\norm{\sK}_{L(C_n^\alpha(\Omega),C_d^\beta(\Omega_1))}
		\leq
		\max\set{M\max\set{1,(\diam\Omega)^\alpha},N} 
	$$
	holds. 
\end{proof}

\begin{cor}[compactness of $\sK$]\label{corkcom}
	A Fredholm operator $\sK\in L(C_p^\alpha(\Omega),C_d^\gamma(\Omega_1))$ is compact, provided one of the following holds:
	\begin{itemize}
		\item[(i)] $\alpha\in(0,1]$ and $\gamma=\beta$, 

		\item[(ii)] $\Omega_1$ is bounded, $\alpha\in(0,1]$ and $\gamma\in[0,\beta]$, 

		\item[(iii)] $\Omega_1$ is compact, $\gamma\in[0,\beta]$ and $\lim_{x\to x_0}\tilde h(x,x_0)=0$ uniformly in $x_0\in\Omega_1$, 

		\item[(iv)] $\Omega_1$ is compact and $\gamma\in[0,\beta)$.
	\end{itemize}
\end{cor}
\begin{proof}
	We write $\sK_0:C_p^\alpha(\Omega)\to C_d^\beta(\Omega_1)$ and $\sK_\alpha^\gamma:C_p^\alpha(\Omega)\to C_d^\gamma(\Omega_1)$ instead of $\sK$.

	(I) Claim: \emph{If (iii) holds, then $\sK_0^\beta\in L(C_n^0(\Omega),C_d^\beta(\Omega_1))$ is compact.}\\
	Given the unit ball $B:=\bar B_1(0,C_n^0(\Omega))$ we apply the compactness criterion from \tref{thmaa} in order to show that $\sK_0^\beta B\subseteq C_d^\beta(\Omega_1)$ is relatively compact. First, $\sK_0^\beta B$ is bounded due to the above \tref{thmkcom}. Second, given $\eps>0$ we obtain by assumption that there exists a $\delta>0$ such that $d(x,\bar x)<\delta$ yields $\tilde h(x,\bar x)<\eps$ for all $x,\bar x\in\Omega_1$. Hence, the assumption $(\bar K_1)$ implies that
	$$
		\abs{(\sK u)(x)-(\sK u)(\bar x)}
		\stackrel{\eqref{kdef}}{\leq}
		\int_\Omega\abs{k(x,y)-k(\bar x,y)}\d\mu(y)
		\stackrel{\eqref{khyp2}}{\leq}
		\eps d(x,\bar x)^\beta\fall x,\bar x\in\Omega_1
	$$
	and all $u\in B$, which guarantees relative compactness of $\sK_0^\beta B$.

	(II) Under (i) the operator $\sK=\sK_0^0\sI_\alpha^0$ is a composition of the continuous $\sK_0^0$ (see \pref{propfreda}) with the compact mapping $\sI_\alpha^0$ (see \tref{thm10}). Under (ii) one has $\sK=\sI_\beta^\gamma\sK_0$ with the bounded embedding $\sI_\beta^\gamma$ (see \tref{thmembed}) and the compact $\sK_0$ (due to (i)). Under the assumptions (iii) we have $\sK=\sI_\beta^\gamma\sK_0^\beta\sI_\alpha^0$ with bounded embeddings and the compact $\sK_0^\beta$ (thanks to step (I)). Finally, in case (iv) one has $\sK=\sI_\beta^\gamma\sK_0$, where $\sK_0$ is continuous and $\sI_\beta^\gamma$ is compact. Since at least one operator in the above compositions is compact, the compactness of $\sK$ results from \cite[p.~417, Thm.~1.2]{lang:93}.
\end{proof}
\subsection{Nemytskii operators}
\label{sec32}
A \emph{Nemytskii operator}\footnote{also denoted as \emph{composition} or \emph{superposition operator}. Further transcriptions are Nemytskij, Nemyzki, Nemytsky, Nemyckij, Nemyckii, Nemitski, Nemitskii, Nemitsky, Nemickij, Nemickii or Niemytzki} is a mapping of the form
\begin{align}
	\sG:U_\alpha&\to F(\Omega,\R^p),&
	\sG(u)(x)&:=g(x,u(x))\fall x\in\Omega,
	\label{gdef}
\end{align}
which is generated by a function $g:\Omega\tm Z\to\R^p$. Our terminology using the letter 'g' comes from \emph{growth function} met in applications \cite{kot:schaeffer:86,lutscher:19}. 
\begin{hypo}
	Let $m\in\N_0$. With $0\leq k\leq m$ one assumes: 
	\begin{itemize}
		\item[$(N_0^k)$] The partial derivative $D_2^kg:\Omega\tm Z\to L_k(\R^n,\R^p)$ exists and can be continuously extended to $\Omega\tm\overline{Z}$.
	\end{itemize}
\end{hypo}
\begin{prop}[well-definedness of $\sG$]\label{propgwell}
	Assume that $(N_0^0)$ holds. Then a Nemytskii operator $\sG:U_\alpha\to C_p^0(\Omega)$ is well-defined, bounded and continuous. Moreover, $\sG$ is completely continuous, provided $\alpha\in(0,1]$. 
\end{prop}
\begin{rem}\label{remcont}
	For $\alpha=0$ and compact intervals $\Omega\subset\R$ a Nemytskii operator $\sG:C^0(\Omega)\to C_p^0(\Omega)$ is well-defined if and only if $g:\Omega\tm\R\to\R^p$ is continuous \cite[Thm.~3.1]{appell:etal:11}. Indeed, due to \cite[Table 1]{appell:etal:11} one has the equivalences:
	\begin{center}
		\begin{tabular}{ccccc}
			&&$g$ is continuous&&\\
			&&$\Updownarrow$&&\\
			$\sG$ is bounded & $\Leftrightarrow$ & $\sG$ is well-defined & $\Leftrightarrow$ & $\sG$ is continuous\\
		\end{tabular}
	\end{center}
\end{rem}
\begin{proof}
	We denote $\sG$ defined on $U_0$ as $\sG_0$. Since every $u\in U_\alpha$ is continuous, also $\sG(u):\Omega\to\R^p$ is continuous and bounded. As a result, $\sG:U_\alpha\to C_p^0(\Omega)$ is well-defined. Let $u_0\in U_\alpha$ and choose $r:=\norm{u_0}_0+1$. Now $g$ can be extended continuously to $\Omega\tm\bar Z$ by assumption $(N_0^0)$ and $g$ is uniformly continuous on $\Omega\tm(Z\cap\bar B_r(0))$. Given $\eps>0$ this means that there exists a $\delta>0$ such that
	\begin{displaymath}
		\begin{cases}
			d(x,\bar x)<\delta,\\
			\abs{z-\bar z}<\delta
		\end{cases}
		\quad\Rightarrow\quad
		\abs{g(x,z)-g(\bar x,\bar z)}<\eps
		\fall x,\bar x\in\Omega,\,z,\bar z\in Z\cap\bar B_r(0).
	\end{displaymath}
	If $u\in U_\alpha\cap B_\delta(u_0,C_n^0(\Omega))$ and $\delta<1$, then $\abs{u(x)}\leq\abs{u_0(x)}+\abs{u(x)-u_0(x)}\leq r$ and consequently
	\begin{displaymath}
		\abs{[\sG(u)-\sG(u_0)](x)}
		=
		\abs{g(x,u(x))-g(x,u_0(x))}
		\leq
		\eps\fall x\in\Omega.
	\end{displaymath}
	Passing to the supremum over $x\in\Omega$ yields $\norm{\sG(u)-\sG(u_0)}_0\leq\eps$, i.e.\ $\sG$ is continuous. The boundedness of $\sG$ results from the uniform continuity of $g$ on $\Omega\tm \overline{Z}$, as well as properties of the norm $\norm{\cdot}_0$. In conclusion, $\sG$ is well-defined, bounded and continuous. 
	
	Thanks to \tref{thm10} the embedding $\sI_\alpha^0$ is compact and therefore $\sG=\sG_0\sI_\alpha^0$ is even completely continuous for $\alpha\in(0,1]$ (see \cite[pp.~25--26, Thm.~2.1(2)]{precup:02}). 
\end{proof}

\begin{hypo}
	Let $\vartheta\in(0,1]$. 
	\begin{itemize}
		\item[$(N_0')$] For every $r>0$ there exists a $l_r'\geq 0$ such that
		\begin{equation}
			\abs{g(x,z)-g(x,\bar z)}\leq l_r'\abs{z-\bar z}^\vartheta
			\fall x\in\Omega,\,z,\bar z\in Z\cap\bar B_r(0). 
			\label{nem02}
		\end{equation}
	\end{itemize}
\end{hypo}
\begin{cor}\label{cornemlip}
	Assume that $g(\cdot,z):\Omega\to\R^p$ is continuous for all $z\in Z$ and $(N_0')$ holds. Then a Nemytskii operator $\sG:U_\alpha\to C_p^0(\Omega)$ is well-defined and H\"older on bounded sets, that is 
	\begin{displaymath}
		\intcc{\sG|_{U_\alpha\cap\bar B_r(0,C_n^0(\Omega))}}_{\vartheta}\leq l_r'\fall r>0.
	\end{displaymath}
\end{cor}
The same argument in case $\sup_{r>0}l_r'<\infty$ yields a global H\"older condition for $\sG:U_\alpha\to C_p^0(\Omega)$.
\begin{rem}
	Note for $\alpha=0$ and a compact interval $\Omega\subset\R$, Nemytskii operators $\sG:C^0(\Omega)\to C_p^0(\Omega)$ satisfy a local (resp.\ global) Lipschitz condition, if and only if $g:\Omega\tm\R\to\R^p$ does in the second variable. In case $p=1$ even the Lipschitz constants (uniformly in $x\in\Omega$) are the same, see \cite[Thm.~1]{appell:81} and \cite[Thm.~3.2]{appell:etal:11}.
\end{rem}
\begin{proof}
	As a consequence of \tref{thmA5}, $g:\Omega\tm Z\to\R^p$ is continuous and using \lref{lemA1} we can show that $g$ satisfies the assumption $(N_0^0)$. Hence, \pref{propgwell} yields that $\sG$ is well-defined. Moreover, if $r>0$ and $u,\bar u\in U_\alpha$ with $\norm{u}_0,\norm{\bar u}_0\leq r$, then
	\begin{displaymath}
		\abs{[\sG(u)-\sG(\bar u)](x)}
		\stackrel{\eqref{gdef}}{=}
		\abs{g(x,u(x))-g(x,\bar u(x))}
		\stackrel{\eqref{nem02}}{\leq}
		l_r'\abs{u(x)-\bar u(x)}^\vartheta
		\leq
		l_r'\norm{u-\bar u}_\alpha^\vartheta\fall x\in\Omega
	\end{displaymath}
	and passing to the supremum over $x\in\Omega$ implies $\norm{\sG(u)-\sG(\bar u)}_0\leq l_r'\norm{u-\bar u}_\alpha^\vartheta$. 
\end{proof}

We next show that the derivative of a Nemytskii operator is a multiplication operator: 
\begin{lem}\label{lemgk}
	Assume that $(N_0^k)$ holds for some $k\in\N$. 
	Then $\sG^k:U_\alpha\to L_k(C_n^\alpha(\Omega),C_p^0(\Omega))$ given by
	\begin{equation}
		(\sG^k(u)v_1\cdots v_k)(x)
		:=
		D_2^kg(x,u(x))v_1(x)\cdots v_k(x)\fall x\in\Omega,\,v_1,\ldots,v_k\in C_n^\alpha(\Omega)
		\label{gkdef}
	\end{equation}
	is well-defined and continuous.
\end{lem}
\begin{proof}
	Let $v_1,\ldots,v_k\in C_n^\alpha(\Omega)$ be given. 
	With $u\in U_\alpha$ also $x\mapsto D_2^kg(x,u(x))v_1(x)\cdots v_k(x)$ is continuous and consequently $\sG^k(u)v_1\cdots v_k\in C_d^0(\Omega)$ holds, i.e.\ $\sG^k$ is well-defined. Let $(u_l)_{l\in\N}$ be a sequence in $U_\alpha$ with $\lim_{l\to\infty}\norm{u_l-u}_0=0$. Choose $r>0$ so large that $u,u_l\in B_r(0,C_n^0(\Omega))$ for all $l\in\N$. Then $D_2^kg$ is uniformly continuous on $\Omega\tm(Z\cap\bar B_r(0))$ and given $\eps>0$, there exists a $\delta>0$ with
	\begin{displaymath}
		\abs{z-\bar z}<\delta
		\quad\Rightarrow\quad
		\abs{D_2^kg(x,z)-D_2g^k(x,\bar z)}<\eps
		\fall x\in\Omega,\,z,\bar z\in Z\cap\bar B_r(0).
	\end{displaymath}
	Moreover, there exists a $L\in\N$ such that $\abs{u_l(x)-u(x)}\leq\delta$ for all $l\geq L$ and consequently
	\begin{eqnarray*}
		\abs{((\sG^k(u_l)-\sG^k(u))v_1\cdots v_k)(x)}
		& \stackrel{\eqref{gkdef}}{=} &
		\abs{\bigl(D_2^kg(x,u_l(x))-D_2^kg(x,u(x))\bigr)v_1(x)\cdots v_k(x)}\\
		& \leq &
		\abs{D_2^kg(x,u_l(x))-D_2^kg(x,u(x))}\abs{v_1(x)}\cdots\abs{v_k(x)}
		\fall x\in\Omega.
	\end{eqnarray*}
	Passing to the supremum over $x\in\Omega$ yields 
	$
		\norm{[\sG^k(u_l)-\sG^k(u)]v_1\cdots v_k}_0
		\leq
		\eps
	$
	for $v_1,\ldots v_k\in\bar B_1(0,C_n^\alpha(\Omega))$ and, in turn, $\bigl\|\sG^k(u)-\sG^k(u_0)\bigr\|_{L_k(C_n^\alpha(\Omega),C_p^0(\Omega))}\leq\eps$ for all $l\geq L$. Since $u\in U_\alpha$ was arbitrary, $\sG^k$ is continuous. 
\end{proof}

\begin{prop}[continuous differentiability of $\sG$]\label{propgder}
	Let $m\in\N$. Assume that $(N_0^k)$ hold for all $0\leq k\leq m$ on a convex set $Z\subseteq\R^n$. Then a Nemytskii operator $\sG:U_\alpha\to C_p^0(\Omega)$ is $m$-times continuously differentiable with $D^k\sG=\sG^k$ for every $1\leq k\leq m$. 
\end{prop}
\begin{proof}
	(I) Let $0\leq k<m$. Thanks to \lref{lemgk} the mappings $\sG^k:U_\alpha\to L_k(C_n^\alpha(\Omega),C_p^0(\Omega))$ are well-defined and continuous. Thus, given $u\in U_\alpha$ and $h\in C_n^\alpha(\Omega)$ with $u+h\in U_\alpha$ the remainders
	\begin{displaymath}
		r_k(h)
		:=
		\sup_{\theta\in[0,1]}\norm{\sG^{k+1}(u+\theta h)-\sG^{k+1}(u)}_{L_{k+1}(C_n^\alpha(\Omega),C_p^0(\Omega))}
	\end{displaymath}
	satisfy $\lim_{h\to 0}r_k(h)=0$. Now we obtain from \cite[p.~341, Thm.~4.2]{lang:93} that
	\begin{align*}
		[\sG^k(u+h)-\sG^k(u)-\sG^{k+1}(u)h](x)
		&\stackrel{\eqref{gkdef}}{=}
		D_2^kg(x,u(x)+h(x))-D_2^kg(x,u(x))-D_2^{k+1}g(x,u(x))h(x)\\
		&=
		\int_0^1\intcc{D_2^{k+1}g(x,u(x)+\theta h(x))-D_2^{k+1}g(x,u(x))}h(x)\d\theta, 
	\end{align*}
	consequently 
	\begin{align*}
		\abs{[\sG^k(u+h)-\sG^k(u)-\sG^{k+1}(u)h](x)}
		&\stackrel{\eqref{gkdef}}{\leq}
		\int_0^1\norm{\sG^{k+1}(u+\theta h)-\sG^{k+1}(u)}_{L_{k+1}(C_n^\alpha(\Omega),C_p^0(\Omega))}\d\theta\norm{h}_0\\
		&\leq
		r_k(h)\norm{h}_\alpha
	\end{align*}
	and after passing to the least upper bound over $x\in\Omega$ it results
	\begin{displaymath}
		\norm{\sG^k(u+h)-\sG^k(u)-\sG^{k+1}(u)h}_{L_k(C_n^\alpha(\Omega),C_p^0(\Omega))}
		\leq 
		r_k(h)\norm{h}_\alpha.
	\end{displaymath}
	This establishes that the mapping $\sG^k$ is differentiable in $u$ with the derivative $\sG^{k+1}(u)$ and, in turn, $\sG^{k+1}$ is continuous due to \lref{lemgk}. 

	(II) Applying step (I) in case $k=0$ shows that $\sG$ is continuously differentiable on $U_\alpha$ with the derivative $\sG^1$. Given this, mathematical induction yields that $\sG:U_\alpha\to C_p^0(\Omega)$ is actually $m$-times continuously differentiable with the derivatives $D^k\sG=\sG^k$ for all $0\leq k\leq m$.
\end{proof}

\begin{rem}[boundedness and continuity of $\sG$]
	(1) The boundedness of Nemytskii operators $\sG$ guaranteed in Props.~\ref{propgwell} means that the $\sG$-images of merely $\norm{\cdot}_0$-bounded subsets $B\subset U_\alpha$ are bounded. In particular, the functions in $B$ are not required to have uniformly bounded H\"older constants. 

	(2) The continuity of the Nemytskii operator $\sU$ in \pref{propgwell}, as well as for the derivatives $D^k\sU$ in \pref{propgder} are to be understood in the following strong form: Already convergence in the domain $U_\alpha$ w.r.t.\ the norm $\norm{\cdot}_0$ is sufficient for convergence of the $\sG$-values in the norm $\norm{\cdot}_0$. A corresponding statement holds in \cref{cornemlip}. 
\end{rem}

In contrast to the above situation, Nemytskii operators $\sG$ behave rather differently when mapping into the space of H\"older functions of exponent $\alpha\in(0,1]$. For instance, \cite[Example 3.10]{appell:etal:11} constructs a discontinuous function $g$ (hence $\sG$ fails to map $C^0[0,1]$ into itself by \rref{remcont}) such that $\sG$ maps $C^\alpha[0,1]$ into itself. Below we survey some properties relating the mappings
\begin{align*} 
	g:\Omega\tm\R&\to\R^p,&
	\sG:C^\alpha(\Omega)&\to C_p^\alpha(\Omega), 
\end{align*}
when $\Omega\subset\R^\kappa$ is compact; one denotes $g$ as \emph{autonomous}, if it does not depend on the first argument:
\begin{itemize}
	\item \emph{Well-definedness and boundedness}: In \cite[Thm.~1.1]{chiappinelli:nugari:95} it is shown that the condition
	\begin{equation}
		\forall r>0\,
		\exists k(r)>0:
		\abs{g(x,z)-g(\bar x,\bar z)}\leq k(r)\intoo{\abs{x-\bar x}^\alpha+\tfrac{\abs{z-\bar z}}{r}}		
		\label{no36}
	\end{equation}
	for all $x,\bar x\in\Omega$, $z,\bar z\in\bar B_r(0)$ is equivalent to $\sG$ being well-defined and bounded (see also \cite[Thm.~7.3]{appell:zabrejko:90}). In comparison, a necessary and sufficient condition for $\sG$ to be merely well-defined is more clumsy, restricted to $\Omega=[a,b]$, and given in terms of (see \cite[Thm.~3.8]{appell:etal:11} or \cite[Thm.~7.1]{appell:zabrejko:90})
	\begin{equation}
		\forall(x_0,z_0)\in\Omega\tm\R\,
		\forall r>0\,
		\exists k(r)>0\,
		\exists \delta>0:
		\abs{g(x,z)-g(\bar x,\bar z)}\leq k(r)\intoo{\abs{x-\bar x}^\alpha+\tfrac{\abs{z-\bar z}}{r}}
		\label{no35}
	\end{equation}
	for all $x,\bar x\in\Omega$, $z,\bar z\in\R$ with $x,\bar x\in B_r(x_0)$, $\abs{z-z_0}\leq r\abs{x-x_0}^\alpha$ and $\abs{\bar z-z_0}\leq r\abs{\bar x-x_0}^\alpha$.\\
	If $g$ is autonomous, then the Lipschitz condition \eqref{nem02} with $\vartheta=1$ is even necessary and sufficient for $\sG$ being well-defined, see \cite[Thm.~1]{drabek:75}.\\
	Let $g$ be autonomous and $\Omega=[a,b]$. Now every well-defined $\sG$ is bounded (see \cite[Cor.~2.1]{goebel:sachweh:99}) and $g$ is continuous (see \cite[Thm.~7.5]{appell:zabrejko:90}). 
	\item \emph{Continuity}: If the partial derivative $D_2g$ exists and satisfies
	\begin{equation}
		\begin{split}
			D_2g:\Omega\tm\R\to\R^p&\text{ is $\alpha$-H\"older in the first argument uniformly in}\\
			&\text{ the second argument from compact subsets of $\R$},
		\end{split}
		\label{no35s}
	\end{equation}
	then \eqref{no36} implies that $\sG$ is continuous (cf.~\cite[Thm.~2.2]{nugari:93}). Conversely, the partial derivative $D_2g$ exists, if $\sG$ is continuous and \eqref{no36} is valid (see \cite[Thm.~2.2]{nugari:93}), or if $\sG$ is bounded and $\Omega=[a,b]$ (see \cite[Cor.~2.3]{nugari:93}). A characterization of $\sG$ being uniformly continuous on bounded sets can be found in \cite[Thm.~2.1]{chiappinelli:nugari:95}. For intervals $\Omega=[a,b]$ it follows from \cite[Table 5]{appell:etal:11} that:
	\begin{center}
		\begin{tabular}{ccccc}
			$g$ satisfies \eqref{no36}&$\Rightarrow$&$g$ satisfies \eqref{no35} &$\Leftarrow$&$g$ is continuously differentiable\\
			$\Updownarrow$&&$\Updownarrow$&&$\Uparrow$\\
			$\sG$ is bounded & $\Rightarrow$ & $\sG$ is well-defined & $\Leftarrow$ & $\sG$ is continuous\\
		\end{tabular}
	\end{center}
	More can be said whenever $g$ is autonomous and $\Omega=[a,b]$: Then reproducing \cite[Table 4]{appell:etal:11} the following implications hold:
	\begin{center}
		\begin{tabular}{ccccc}
			&&$g$ is Lipschitz on bounded sets&$\Leftarrow$&$g$ is continuously differentiable\\
			&&$\Updownarrow$&&$\Updownarrow$\\
			$\sG$ is bounded & $\Leftrightarrow$ & $\sG$ is well-defined & $\Leftarrow$ & $\sG$ is continuous\\
		\end{tabular}
	\end{center}

	\item \emph{Lipschitz condition}: It is shown in \cite[Thm.~3.1]{chiappinelli:nugari:95} that $\sG$ is Lipschitz on bounded sets if and only if both $g$ and $D_2g$ satisfy an estimate \eqref{no36}. The necessity to assume the existence of the partial derivatives also arose for Urysohn operators (see \cref{cor25}). Yet, the assumption of a global Lipschitz condition for $\sG$ leads to a degeneracy in $g$. Indeed, \cite{matkowska:84} shows that $\sG$ is globally Lipschitz, if and only if all components $g_1,\ldots,g_p:\Omega\tm\R\to\R$ of $g$ are affine linear, i.e.\ there exist $\bar a_i,\bar b_i\in C^\alpha(\Omega)$ with 
	\begin{equation}
		g_i(x,z)=z\bar a_i(x)+\bar b_i(x)\fall 1\leq i\leq p,\,x\in\Omega,\, z\in\R.
		\label{degen}
	\end{equation}
	Yet, even uniform continuity of $\sG$ is sufficient for \eqref{degen} to hold (see \cite[Thm.~2]{matkowski:09}). Nonetheless, if $\sI_\alpha^\beta\sG$ satisfies a global Lipschitz condition, then $\sup_{x\in\Omega}[g(x,\cdot)]_1<\infty$ (cf.\ \cite[Thm.~2.5]{appell:etal:11} and \eqref{degen} holds in case $\beta=1$ (cf.\ \cite[Thm.~2.6]{appell:etal:11}).\\ 
	Let $g$ be autonomous and $\Omega=[a,b]$: Then a well-defined Nemytskii operator $\sG$ is globally Lipschitz, if and only if $g:\R\to\R^p$ is affine-linear, i.e.\ there exist $\bar a,\bar b\in\R^p$ such that $g(z)=z\bar a+\bar b$ (see \cite[Thm.~2.3(b)]{goebel:sachweh:99}). 

	\item \emph{Continuous differentiability}: If $g(x,\cdot)$ is twice differentiable such that $g,D_2g$ satisfy \eqref{no35} and \eqref{no35s} holds with $D_2^2g$ (instead of $D_2g$), then $\sG$ is continuously differentiable (cf.~\cite[Thm.~4.1]{nugari:93}). Note that also for the continuous differentiability of Urysohn operators we needed to assume that the second order derivative of the kernel function exists (cf.\ \tref{thmuderb}). For a characterization of $\sG$ being continuously differentiable with a derivative being uniformly continuous on bounded sets we refer to \cite[Thm.~4.1]{chiappinelli:nugari:95}. In case $\Omega=[a,b]$, then differentiability of $\sI_\alpha^\beta\sG$ is characterized in \cite[Thm.~7.11]{appell:zabrejko:90}. Furthermore, if $\sG$ is differentiable with a globally bounded derivative, then $g$ is affine linear, i.e.\ $g$ satisfies \eqref{degen}.\\
	For autonomous $g$ and $\Omega=[a,b]$ an elegant characterization holds: The Nemytskii operator $\sG$ is continuously differentiable, if and only if $g:\R\to\R^p$ is twice continuously differentiable (see \cite[Thm.~2.4]{goebel:sachweh:99}).
\end{itemize}
\subsection{Hammerstein operators}
In the following, we understand \emph{Hammerstein operators} \eqref{hdef} as composition
$$
	\sH=\sK\sG:C_n^\alpha(\Omega)\to F(\Omega_1,\R^d)
$$
of Fredholm operators $\sK\in L(C_p^0(\Omega),C_d^\beta(\Omega_1))$ and Nemytskii operators $\sG:C_n^\alpha(\Omega)\to C_p^0(\Omega)$ given in \eqref{kdef} resp.\ \eqref{gdef}. Hence, our above preparations immediately yield properties of~$\sH$: 
\begin{thm}[well-definedness of $\sH$]
	Assume that $(\bar K_1,K_2)$ and $(N_0^0)$ hold. Then a Hammerstein operator $\sH:U_\alpha\to C_d^\beta(\Omega_1)$ is well-defined, bounded and continuous.
\end{thm}
\begin{proof}
	As composition of $\sG:C_n^\alpha(\Omega)\to C_p^0(\Omega)$ and $\sK\in L(C_p^0(\Omega),C_d^\beta(\Omega_1))$, the claims for $\sH=\sK\sG$ result directly from \tref{thmkcom} (with $\alpha=0$) and \pref{propgwell}. 
\end{proof}

\begin{cor}[complete continuity of $\sH$]
	A Hammerstein operator $\sH:U_\alpha\to C_d^\gamma(\Omega_1)$ is completely continuous, provided one of the following holds:
	\begin{itemize}
		\item[(i)] $\alpha\in(0,1]$ and $\gamma=\beta$, 

		\item[(ii)] $\Omega_1$ is bounded, $\alpha\in(0,1]$ and $\gamma\in[0,\beta]$, 

		\item[(iii)] $\Omega_1$ is compact, $\gamma\in[0,\beta]$ and $\lim_{x\to x_0}\tilde h(x,x_0)=0$ uniformly in $x_0\in\Omega_1$, 

		\item[(iv)] $\Omega_1$ is compact and $\gamma\in[0,\beta)$.
	\end{itemize}
\end{cor}
\begin{proof}
	It results from \pref{propgwell} (case (i)) and \cref{corkcom} (cases (ii--iii)) that at least one of the functions in the composition $\sH=\sK\sG$ is completely continuous. 
\end{proof}

\begin{thm}[continuous differentiability of $\sH$]
	Let $m\in\N$. Assume that $(\bar K_1,K_2)$ and $(N_0^k)$ hold for all $0\leq k\leq m$ on a convex set $Z\subseteq\R^n$. Then a Hammerstein operator $\sH:U_\alpha\to C_d^\beta(\Omega_1)$ is $m$-times continuously differentiable with $D^k\sH=\sK\sG^k$ for every $1\leq k\leq m$. 
\end{thm}
\begin{proof}
	This results from the chain rule \cite[p.~337]{lang:93}, \tref{thmkcom} (with $\alpha=0$) and \pref{propgder}. 
\end{proof}

\begin{rem}[convolutive Hammerstein operators]
	Suppose a growth function $g:[a,b]\tm Z\to\R^p$ generates a Nemytskii operator $\sG$ mapping into $C_p^\alpha[a,b]$. Then the smoothing properties from \sref{sec23} extend to \emph{convo\-lutive Hammerstein operators} $\tilde\sK u(v):=\int_a^b\tilde k(x-y)g(y,u(y))\d y$ with an ambient kernel $\tilde k:[a-b,b-a]\to\R^{d\tm p}$. 
\end{rem}
\begin{appendix}
\section{H\"older continuous functions}
\label{appA}
The definition of continuity for a function $u$ in e.g.\ a point $x_0$ is not quantitative in the sense that it provides no information on how fast its values $u(x)$ approach $u(x_0)$ as $x\to x_0$. In consequence, the modulus of continuity $\omega:\R_+\to\R_+$ of a continuous function $u$ satisfying the estimate $\norm{u(x)-u(x_0)}\leq\omega(d(x,x_0))$ may decrease arbitrarily slowly. For this reason, the space of continuous functions is often not suitable for a quantitative analysis in fields such as numerical analysis or partial differential equations. A straightforward and feasible way to strengthen the notion of continuity of $u$ is to assume that its modulus of continuity is proportional to a power of $d(x,x_0)^\alpha$ for some exponent $\alpha\in(0,1]$. Such functions are denoted as H\"older continuous and in the focus of this appendix. 

Our overall setting is as follows. Let $(\Omega,d)$ be a metric space and $(Y,\norm{\cdot})$ be a normed space over $\K$, which stands for the real or complex field. 

A function $u:\Omega\to Y$ is said to be \emph{$\alpha$-H\"older} (with \emph{H\"older exponent} $\alpha\in(0,1]$), if it satisfies
\begin{displaymath}
	[u]_\alpha:=\sup_{\substack{x,\bar x\in\Omega\\ x\neq\bar x}}
	\frac{\norm{u(x)-u(\bar x)}}{d(x,\bar x)^\alpha}<\infty;
\end{displaymath}
the finite quantity $[u]_\alpha\geq 0$ is called \emph{H\"older constant} of $u$. One speaks of a \emph{H\"older continuous} function $u$, if it is $\alpha$-H\"older for some $\alpha\in(0,1)$ and in case $\alpha=1$ one denotes $u$ as \emph{Lipschitz continuous} with \emph{Lipschitz constant} $[u]_1$ --- a comprehensive approach to this class of functions is given in \cite{cobzas:etal:19}. For the set of all such $\alpha$-H\"older functions $u:\Omega\to Y$ we write $C^\alpha(\Omega,Y)$, supplemented by $C^0(\Omega,Y)$ for the linear space of continuous functions. It is convenient to denote a continuous function as $0$-H\"older, and unless indicated otherwise, let us assume $\alpha\in[0,1]$ throughout. 
\begin{rem}
	(1) A function $u:\Omega\to Y$ is constant, if and only if its H\"older constant vanishes.

	(2) For $\alpha\in(0,1]$ there is an evident relation between H\"older functions and Lipschitz functions: Indeed, $u:(\Omega,d)\to Y$ is $\alpha$-H\"older, if and only if $u:(\Omega,d_\alpha)\to Y$ is Lipschitz with the metric $d_\alpha:\Omega\tm\Omega\to\R_+$ given by $d_\alpha(x,\bar x):=d(x,\bar x)^\alpha$. Of course the metrics $d$ and $d_\alpha$ on $\Omega$ are not equivalent. 

	(3) One does restrict to exponents $\alpha\in(0,1]$ for the following reason. Suppose that a function $u:\Omega\to Y$ on an open subset $\Omega\subseteq\R^\kappa$ satisfies $[u]_\alpha<\infty$ for an exponent $\alpha>1$. Then 
	\begin{displaymath}
		\norm{u(x)-u(\bar x)-0(x-\bar x)}
		=
		\norm{u(x)-u(\bar x)}
		\leq
		[u]_\alpha\abs{x-\bar x}^{\alpha-1}\abs{x-\bar x}\fall x,\bar x\in\Omega
	\end{displaymath}
	yields that $u$ is differentiable on $\Omega$ with derivative $0$ and thus constant on the components of $\Omega$. 

	(4) Suppose that $\Omega$ has a finite, positive diameter and that $\phi:[0,1]\to\R_+$ is a function satisfying $\phi(0)=0$, $\phi(1)=1$ such that $t\mapsto\phi(t)$ and $t\mapsto\tfrac{t}{\phi(t)}$ are positive and increasing on $(0,1)$. Then $u:\Omega\to Y$ is called \emph{generalized H\"older}, if
	$
		\sup_{x,\bar x\in\Omega,x\neq\bar x}
		\norm{u(x)-u(\bar x)}\phi\bigl(\tfrac{d(x,\bar x)}{\diam\Omega}\bigr)^{-1}
		<\infty
	$
	(see \cite[Ch.~7]{appell:zabrejko:90} and \cite{banas:nalepa:16}). In case $\phi(t):=t^{\alpha}$, $\alpha\in(0,1]$ this reduces to the situation studied here. 

	(5) Differentiable functions on $\Omega\subseteq\R^\kappa$ whose $m$th derivative satisfies an $\alpha$-H\"older condition, and associated function spaces $C^{m,\alpha}(\Omega,Y)$, are addressed in \cite[pp.~30ff, Sect.~1.5]{fiorenza:16} or \cite[pp.~51ff, Sect.~4.1]{gilbarg:trudinger:01}. 

	(6) The inner structure of H\"older spaces from an abstract Banach spaces perspective is studied in \cite{kalton:04}. 
\end{rem}

\begin{SCfigure}[2]
	\includegraphics[scale=0.5]{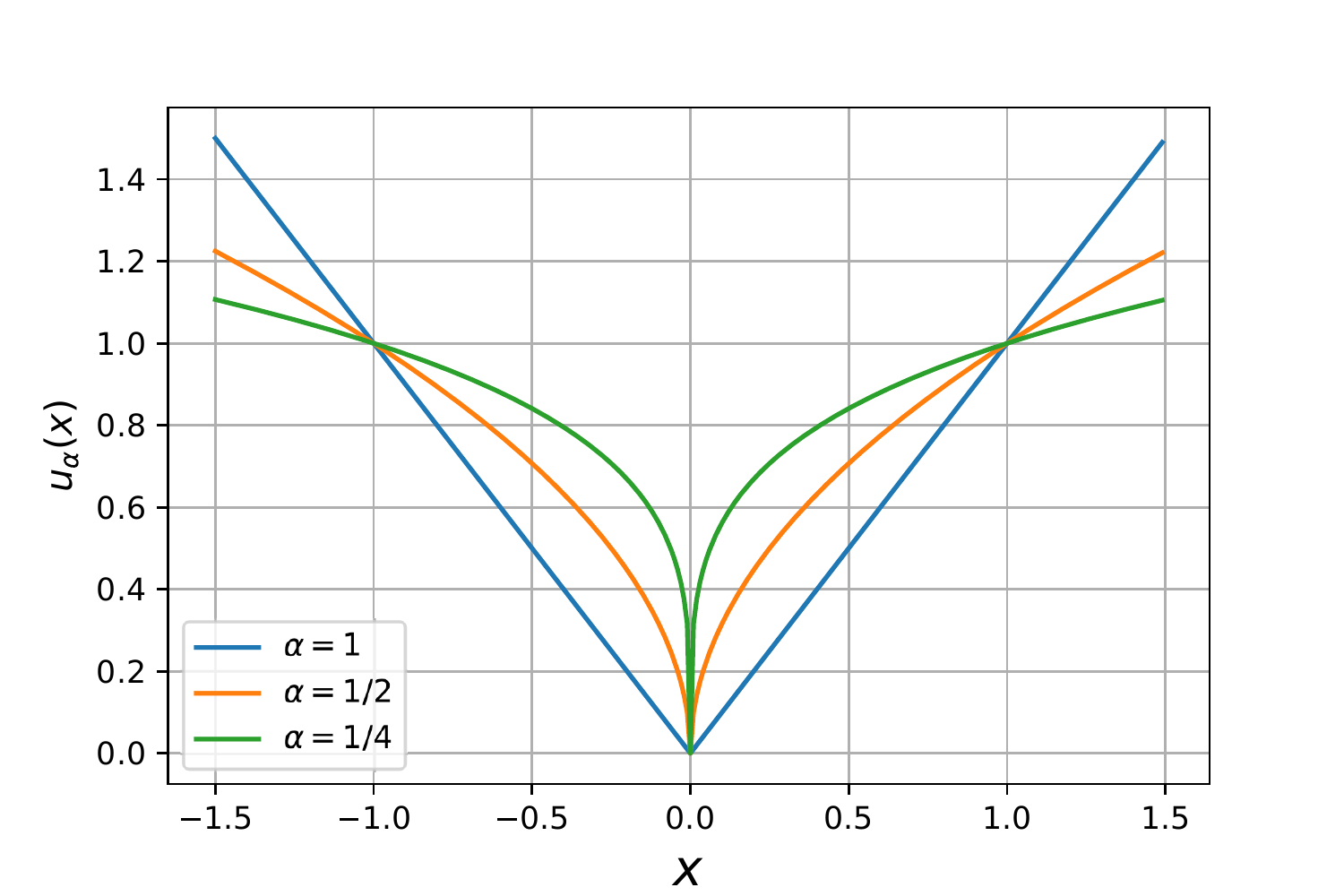}
	\caption{Graphs of the H\"older functions $u_\alpha:\R\to\R$ given by $u_\alpha(x):=|x|^\alpha$ from \eref{exhoelder} for $\alpha\in\set{\tfrac{1}{4},\tfrac{1}{2},1}$. Their decay to $0$ as $x\to 0$ is faster for decreasing values of the H\"older exponent $\alpha$}
	\label{fighoelder}
\end{SCfigure}
\begin{ex}\label{exhoelder}
	Let $\Omega=Y=\R$ and $\alpha\in(0,1]$. The function $u_\alpha:\R\to\R$, $u_\alpha(x):=\abs{x}^\alpha$ is not differentiable in $0$ (see \fref{fighoelder}), but $\alpha$-H\"older. Indeed, given $x,\bar x\in\R$ the case $0<|x|\leq|\bar x|$ leads to the inequality $1-\tfrac{|x|^\alpha}{|\bar x|^\alpha}\leq 1-\tfrac{|x|}{|\bar x|}\leq\Bigl(1-\tfrac{|x|}{|\bar x|}\Bigr)^\alpha$ and therefore $|u_\alpha(\bar x)-u_\alpha(x)|=u_\alpha(\bar x)-u_\alpha(x)\leq\abs{\bar x-x}^\alpha$. A similar argument in case $0<|\bar x|\leq|x|$ yields the assertion with H\"older constant $[u_\alpha]_\alpha\leq 1$. 
\end{ex}

\begin{thm}[local H\"older continuity]
	Let $\Omega$ be compact and $\alpha\in(0,1]$. If $u:\Omega\to Y$ is \emph{locally $\alpha$-H\"older}, i.e.\ every $x\in\Omega$ has a neighborhood $U\subseteq\Omega$ such that $[u|_U]_\alpha<\infty$ holds, then $u$ is $\alpha$-H\"older. 
\end{thm}
\begin{proof}
	We proceed indirectly and suppose that for each $c\geq 0$ there exist $x,\bar x\in\Omega$ yielding the inequality $\norm{u(x)-u(\bar x)}>cd(x,\bar x)^\alpha$. In particular, there are sequences $(x_l)_{l\in\N}$, $(\bar x_l)_{l\in\N}$ in $\Omega$ so that
	\begin{equation}
		\norm{u(x_l)-u(\bar x_l)}>ld(x_l,\bar x_l)^\alpha\fall l\in\N
		\label{aa234}
	\end{equation}
	holds. Since $\Omega$ is compact, w.l.o.g.\ we can assume that these sequences converge to $x^\ast,\bar x^\ast\in\Omega$, respectively. Passing to the limit $l\to\infty$ in \eqref{aa234} therefore enforces $\lim_{l\to\infty}d(x_l,\bar x_l)^\alpha=0$, i.e.\ one has $x^\ast=\bar x^\ast$. Because the function $u$ is assumed to be locally H\"older, there exists a neighborhood $U\subseteq\Omega$ of $x^\ast$ and a real $C\geq 0$ with $\norm{u(x)-u(\bar x)}\leq Cd(x,\bar x)^\alpha$ for all $x,\bar x\in U$. Thanks to the inclusion $x_l,\bar x_l\in U$ for sufficiently large $l\in\N$ this implies $\norm{u(x_l)-u(\bar x_l)}\leq Cd(x_l,\bar x_l)^\alpha$ contradicting \eqref{aa234} for $l>C$. 
\end{proof}

The relationship between differentiable and H\"older continuous functions is addressed in
\begin{ex}\label{exA2}
	(1) Differentiable functions $u:\Omega\to Y$ on open, bounded and convex sets $\Omega\subset\R^\kappa$ having a bounded derivative $Du:\Omega\to L(\R^\kappa,Y)$ are $\alpha$-H\"older for each $\alpha\in(0,1]$. This follows from the Mean Value Inequality \cite[p.~35, Thm.~4.1]{martin:76}, because
	\begin{displaymath}
	\norm{u(x)-u(\bar x)}
	\leq
	\sup_{\xi\in\Omega}\norm{Du(\xi)}\abs{x-\bar x}
	\leq
	(\diam\Omega)^{1-\alpha}\sup_{\xi\in\Omega}\norm{Du(\xi)}\abs{x-\bar x}^\alpha
	\fall x,\bar x\in\Omega
	\end{displaymath}
	and thus $[u]_\alpha\leq(\diam\Omega)^{1-\alpha}\sup_{\xi\in\Omega}\norm{Du(\xi)}$. A version of this result on not necessarily convex sets $\Omega$ can be found in \cite[p.~11, Prop.~1.1.13]{fiorenza:16}. In the Lipschitz case $\alpha=1$ boundedness of $\Omega$ is not required. 

	(2) Rademacher's theorem \cite[p.~414, Thm.~21.2]{dibenedetto:16} guarantees that Lipschitz functions $u:\Omega\to\R^d$ on open sets $\Omega\subseteq\R^\kappa$ are differentiable Lebesgue-almost everywhere in $\Omega$. This situation changes for exponents $\alpha\in(0,1)$ and \cite{hardy:16} shows that the \emph{Weierstra{\ss} function} $u:\R\to\R$ given as Fourier series
	\begin{displaymath}
		u(x)=\sum_{k=0}^\infty a^k\cos(b^k\pi x)
		\quad\text{with $a\in(0,1)$ and integers }b>1
	\end{displaymath}
	satisfying $ab>1+\tfrac{3\pi}{2}$ is $\alpha$-H\"older with exponent $\alpha=-\log_ba$, but nowhere differentiable. 
\end{ex}

\begin{thm}\label{thm1}
	H\"older continuous functions are uniformly continuous. 
\end{thm}
\begin{proof}
	Let $\eps>0$ and $u\in C^\alpha(\Omega,Y)$ (w.l.o.g.\ $u$ is not constant). Setting $\delta:=\bigl(\tfrac{\eps}{[u]_\alpha}\bigr)^{1/\alpha}$ guarantees 
	\begin{displaymath}
		d(x,\bar x)<\delta
		\quad\Rightarrow\quad
		\norm{u(x)-u(\bar x)}\leq[u]_\alpha d(x,\bar x)^\alpha\leq\eps\fall x,\bar x\in\Omega
	\end{displaymath}
	and thus $u$ is uniformly continuous. 
\end{proof}
The converse to \tref{thm1} does not hold.
\begin{SCfigure}[2]
	\includegraphics[scale=0.5]{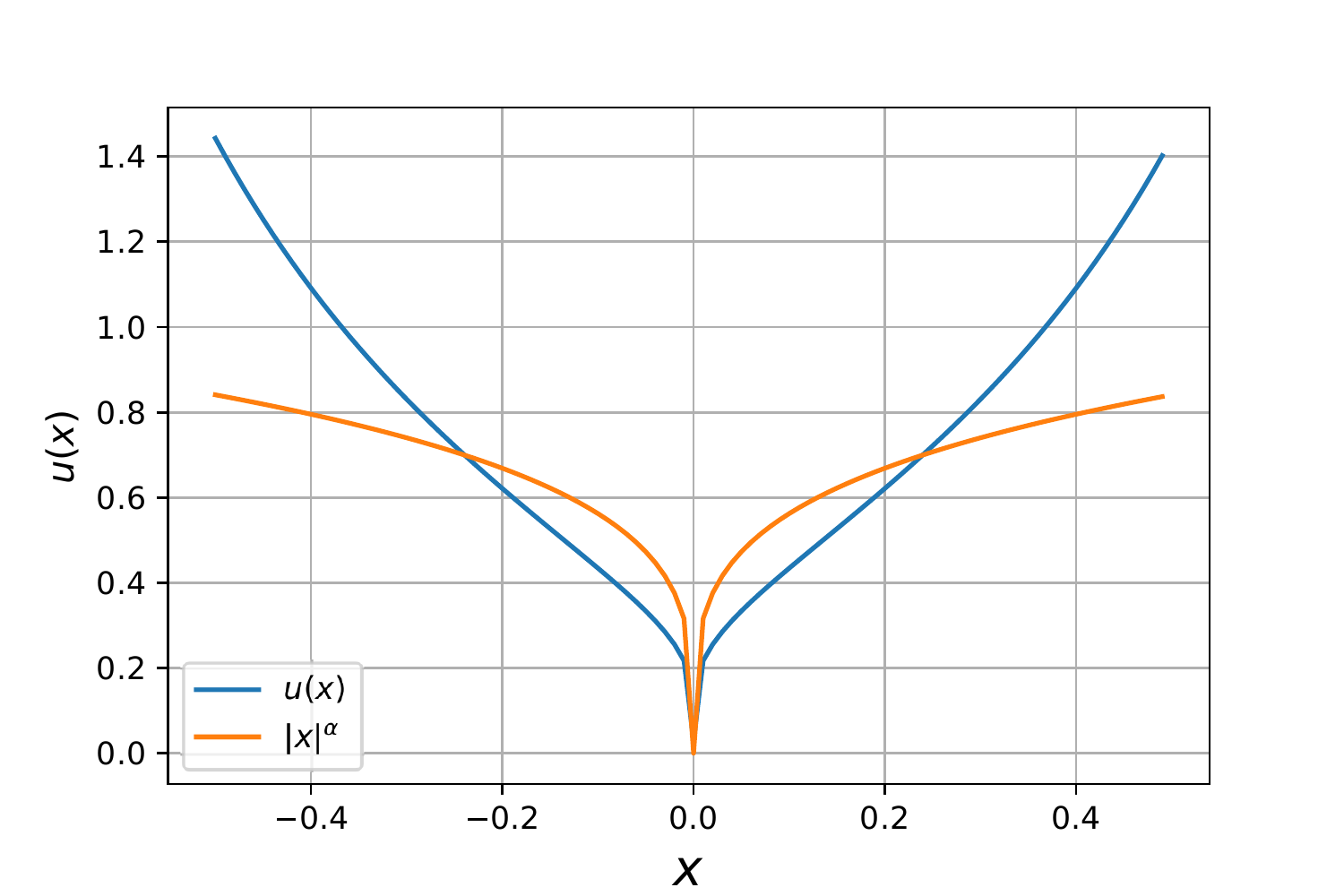}
	\caption{Graphs of the function $u:\R\to\R$ (blue) from \eref{exnothoelder}, which decays to $0$ as $x\to 0$ faster than any H\"older function (orange)}
	\label{fignothoelder}
\end{SCfigure}
\begin{ex}\label{exnothoelder}
	Let $\Omega=[-\tfrac{1}{2},\tfrac{1}{2}]$ and $Y=\R$. The continuous function (see \fref{fignothoelder})
	\begin{align*}
		u:[-\tfrac{1}{2},\tfrac{1}{2}]&\to\R,&
		u(x)&:=
		\begin{cases}
			-\tfrac{1}{\ln|x|},&x\in[-\tfrac{1}{2},\tfrac{1}{2}]\setminus\set{0},\\
			0,&x=0
		\end{cases}
	\end{align*}
	on the compact interval $[-\tfrac{1}{2},\tfrac{1}{2}]$ is uniformly continuous. However, it is not H\"older continuous, because otherwise there would exist reals $\alpha\in(0,1]$, $C\geq 0$ such that $-\tfrac{1}{\ln x}=\abs{u(x)-u(0)}\leq C\abs{x-0}^\alpha=Cx^\alpha$ for all $x\in(0,\tfrac{1}{2}]$ leading to the contradiction
	$
		C\geq-\frac{1}{x^\alpha\ln x}
		\xrightarrow[x\searrow 0]{}\infty
	$
	by the l'Hospital rule. 
\end{ex}

\begin{lem}\label{lemA1}
	Let $\alpha\in(0,1]$ and $Y$ be a Banach space. Every $\alpha$-H\"older function $u:U\to Y$ on an open set $U\subseteq\Omega$ allows an $\alpha$-H\"older extension to the closure $\overline{U}$ with the same H\"older constant. 
\end{lem}
\begin{proof}
	Let $u:U\subseteq\Omega\to Y$ be $\alpha$-H\"older. By \tref{thm1} it follows that $u:U\to Y$ is uniformly continuous. Hence, because $Y$ is complete, there exists a continuous extension $\bar u:\bar U\subseteq\Omega\to Y$ of $u$ to the boundary. We next show that $\bar u$ is $\alpha$-H\"older with $[\bar u]_\alpha=[u]_\alpha$. Thereto, given arbitrary $y,\bar y\in\bar U$ and the estimate $\norm{u(x)-u(\bar x)}\leq[u]_\alpha d(x,\bar x)^\alpha$ for all $x,\bar x\in U$, we first pass to the limit $x\to y$, then to $\bar x\to\bar y$, and it results from the extension property that $\norm{\bar u(y)-\bar u(\bar y)}\leq[u]_\alpha d(y,\bar y)^\alpha$ for all $y,\bar y\in\bar U$. This guarantees that $\bar u$ is $\alpha$-H\"older with $[\bar u]_\alpha\leq[u]_\alpha$. Moreover, it evidently holds
	\begin{displaymath}
		\norm{u(x)-u(\bar x)}
		=
		\norm{\bar u(x)-\bar u(\bar x)}
		\leq
		[\bar u]_\alpha d(x,\bar x)^\alpha\fall x,\bar x\in U
	\end{displaymath}
	and $[u]_\alpha\leq[\bar u]_\alpha$. Thus, the minimum H\"older coefficients of $u$ and $\bar u$ are equal yielding the claim. 
\end{proof}

\begin{thm}\label{thmA2}
	H\"older continuous functions are bounded (on bounded sets). 
\end{thm}
\begin{proof}
	Let $u\in C^\alpha(\Omega,Y)$, $B\subseteq\Omega$ be bounded and choose a fixed $x_0\in B$. Then
	\begin{equation}
		\norm{u(x)}
		\leq
		\norm{u(x_0)}+\norm{u(x)-u(x_0)}
		\leq
		\norm{u(x_0)}+[u]_\alpha d(x,x_0)^\alpha
		\leq
		\norm{u(x_0)}+[u]_\alpha(\diam B)^\alpha
		\label{appstar}
	\end{equation}
	for all $x\in\Omega$ holds and therefore the image $u(B)\subseteq Y$ is bounded. 
\end{proof}

For globally bounded $\alpha$-H\"older functions $u:\Omega\to Y$, we define
\begin{displaymath}
	\norm{u}_\alpha:=
	\begin{cases}
		\sup_{x\in\Omega}\norm{u(x)},&\alpha=0,\\
		\max\set{\sup_{x\in\Omega}\norm{u(x)},[u]_\alpha},&\alpha\in(0,1].
	\end{cases}
\end{displaymath}
On the product $Y_1\tm Y_2$ of two normed spaces $Y_1,Y_2$ we use the product norm
\begin{displaymath}
	\norm{(y_1,y_2)}=\max\set{\norm{y_1},\norm{y_2}}\fall y_1\in Y_1,\,y_2\in Y_2. 
\end{displaymath}
\begin{thm}
	A function $u=(u_1,u_2):\Omega\to Y_1\tm Y_2$ is $\alpha$-H\"older, if and only if both component functions $u_j:\Omega\to Y_j$ are $\alpha$-H\"older for $j=1,2$. In this case and for $\alpha\in(0,1]$ one has
	$
		[u_j]_\alpha\leq[u]_\alpha\leq\max\set{[u_1]_\alpha,[u_2]_\alpha}
	$
	and for globally bounded functions $u$ results
	$
		\|u_j\|_\alpha\leq\norm{u}_\alpha\leq\max\set{\norm{u_1}_\alpha,\norm{u_2}_\alpha}
	$
	for all $j=1,2$. 
\end{thm}
\begin{proof}
	We restrict to the case $\alpha\in(0,1]$.\\ 
	$(\Rightarrow)$ If $u:\Omega\to Y_1\tm Y_2$ is $\alpha$-H\"older, then also the components $u_1,u_2$ are $\alpha$-H\"older due to
	$$
		\norm{u_j(x)-u_j(\bar x)}
		\leq
		\max_{i=1}^2\norm{u_i(x)-u_i(\bar x)}
		=
		\norm{u(x)-u(\bar x)}
		\leq
		[u]_\alpha d(x,\bar x)^\alpha\fall x,\bar x\in\Omega,\,j=1,2.
	$$
	$(\Leftarrow)$ Conversely, if the component functions $u_1,u_2$ are $\alpha$-H\"older, then also $u$ is $\alpha$-H\"older because of
	$$
		\norm{u(x)-u(\bar x)}
		=
		\max_{i=1}^2\norm{u_i(x)-u_i(\bar x)}
		\leq
		\max_{i=1}^2[u_i]_\alpha d(x,\bar x)^\alpha\fall x,\bar x\in\Omega.
	$$
	These inequalities also imply the claimed estimates for the H\"older constants. Combining them with
	$$
		\norm{u_j(x)}\leq\norm{u(x)}=\max_{i=1}^2\norm{u_i(x)}\leq\max_{i=1}^2\norm{u_i}_\alpha
		\fall x\in\Omega,\,j=1,2
	$$
	yields the estimates for $\norm{\cdot}_\alpha$ after passing to the supremum over $x\in\Omega$. 
\end{proof}

The product of two metric spaces $\Omega_1,\Omega_2$ is equipped with the product metric
$$
	d((x_1,x_2),(\bar x_1,\bar x_2))
	:=
	\max\set{d(x_1,\bar x_1),d(x_2,\bar x_2)}\fall x_1,\bar x_1\in\Omega_1,\,x_2,\bar x_2\in\Omega_2.
$$
\begin{thm}
	Let $\alpha\in(0,1]$. A function $u:\Omega_1\tm\Omega_2\to Y$ defined on the product of metric spaces $\Omega_1,\Omega_2$ is $\alpha$-H\"older, if and only if all functions $u(\cdot,x_2):\Omega_1\to Y$ and $u(x_1,\cdot):\Omega_2\to Y$ are $\alpha$-H\"older uniformly in $x_2\in\Omega_2$ resp.\ $x_1\in\Omega_1$. 
\end{thm}
\begin{proof}
	Suppose that $x_1,\bar x_1\in\Omega_1$ and $x_2,\bar x_2\in\Omega_2$ are arbitrary.\\
	$(\Rightarrow)$ Let $u:\Omega_1\tm\Omega_2\to Y$ be $\alpha$-H\"older. From the estimates
	\begin{align*}
		\norm{u(x_1,x_2)-u(\bar x_1,x_2)}
		&\leq
		[u]_\alpha d(x_1,\bar x_1)^\alpha,&
		\norm{u(x_1,x_2)-u(x_1,\bar x_2)}
		&\leq
		[u]_\alpha d(x_2,\bar x_2)^\alpha
	\end{align*}
	one deduces that $u(\cdot,x_2)$ and $u(x_1,\cdot)$ are $\alpha$-H\"older (uniformly in $x_2$ resp.\ $x_1$).\\
	$(\Leftarrow)$ Conversely, the estimate
	\begin{align*}
		\norm{u(x_1,x_2)-u(\bar x_1,\bar x_2)}
		&\leq
		\norm{u(x_1,x_2)-u(\bar x_1,x_2)}+\norm{u(\bar x_1,x_2)-u(\bar x_1,\bar x_2)}\\
		&\leq
		\sup_{x\in\Omega_2}[u(\cdot,x)]_\alpha d(x_1,\bar x_1)^\alpha+
		\sup_{x\in\Omega_1}[u(x,\cdot)]_\alpha d(x_2,\bar x_2)^\alpha\\
		&\leq
		\Bigl(\sup_{x\in\Omega_2}[u(\cdot,x)]_\alpha+\sup_{x\in\Omega}[u(x,\cdot)]_\alpha\Bigr)
		\max\set{d(x_1,\bar x_1),d(x_2,\bar x_2)}^\alpha
	\end{align*}
	implies a H\"older condition for $u$. 
\end{proof}

\begin{thm}\label{thmA5}
	Let $\alpha\in(0,1]$. If a function $u:\Omega_1\tm\Omega_2\to Y$ satisfies
	\begin{itemize}
		\item[(i)] $\sup_{x_1\in\Omega_1}[u(x_1,\cdot)]_\alpha<\infty$, 

		\item[(ii)] $u(\cdot,x_2):\Omega_1\to Y$ is continuous for all $x_2\in\Omega_2$,
	\end{itemize}
	then $u$ is continuous. 
\end{thm}
\begin{proof}
	Let $(x_1^\ast,x_2^\ast)\in\Omega_1\times\Omega_2$ be the limit of sequences $(x_l^1)_{l\in\N}$, $(x_l^2)_{l\in\N}$ in the respective metric spaces $\Omega_1$ and $\Omega_2$. If $k_2:=\sup_{x\in\Omega_1}[u(x,\cdot)]_\alpha$, then
	\begin{align*}
		0&\leq
		\norm{u(x_l^1,x_l^2)-u(x_1^\ast,x_2^\ast)}
		\leq
		\norm{u(x_l^1,x_l^2)-u(x_l^1,x_2^\ast)}+\norm{u(x_l^1,x_2^\ast)-u(x_1^\ast,x_2^\ast)}\\
		&\stackrel{(i)}{\leq}
		k_2d(x_l^2,x_2^\ast)^\alpha+\norm{u(x_l^1,x_2^\ast)-u(x_1^\ast,x_2^\ast)}
		\xrightarrow[l\to\infty]{(ii)}0
	\end{align*}
	establishes the continuity of $u$, since $(x_1^\ast,x_2^\ast)$ were arbitrary. 
\end{proof}

Let us next investigate the algebraic structure of the space of $\alpha$-H\"older functions. 
\begin{thm}[sum rule]\label{thmsum}
	With functions $u_1,u_2:\Omega\to Y$ also $\lambda_1u_1+\lambda_2u_2$ is $\alpha$-H\"older for all $\lambda_1,\lambda_2\in\K$. In case $\alpha\in(0,1]$ one has
	$
		[\lambda_1u_1+\lambda_2u_2]_\alpha
		\leq
		\abs{\lambda_1}[u_1]_\alpha+\abs{\lambda_2}[u_2]_\alpha
	$
	and for globally bounded $u_1,u_2$ holds
	\begin{displaymath}
		\norm{\lambda_1u_1+\lambda_2u_2}_\alpha
		\leq
		\abs{\lambda_1}\norm{u_1}_\alpha+\abs{\lambda_2}\norm{u_2}_\alpha
		\fall\lambda_1,\lambda_2\in\K. 
	\end{displaymath}
\end{thm}
\begin{proof}
	The straightforward proof is left to the reader. 
\end{proof}

A mapping $\cdot:Y_1\tm Y_2\to Y$ is called a \emph{product}, if there exists a constant $C\geq 0$ such that
\begin{align*}
	y_1\cdot(y_2+\bar y_2)&=y_1\cdot y_2+y_1\cdot\bar y_2,&
	(y_1+\bar y_1)\cdot y_2&=y_1\cdot y_2+\bar y_1\cdot y_2,&
	\norm{y_1\cdot y_2}&\leq C\norm{y_1}\norm{y_2}
\end{align*}
for all $y_1,\bar y_1\in Y_1$, $y_2,\bar y_2\in Y_2$. 
\begin{thm}[product rule]
	If the functions $u_1:\Omega\to Y_1$ and $u_2:\Omega\to Y_2$ are globally bounded and $\alpha$-H\"older, then also their product $u_1\cdot u_2:\Omega\to Y$ is $\alpha$-H\"older. In case $\alpha\in(0,1]$ one has the estimates
	$
		[u_1\cdot u_2]_\alpha
		\leq
		C\bigl(\norm{u_2}_0[u_1]_\alpha+\norm{u_1}_0[u_2]_\alpha\bigr)
	$
	and
	$
		\norm{u_1\cdot u_2}_\alpha
		\leq
		C\bigl(\norm{u_2}_0\norm{u_1}_\alpha+\norm{u_1}_0\norm{u_2}_\alpha\bigr)
	$. 
\end{thm}
\begin{proof}
	We restrict to $\alpha\in(0,1]$. Using properties of a product, we obtain from the triangle inequality
	\begin{align*}
		\norm{(u_1\cdot u_2)(x)-(u_1\cdot u_2)(\bar x)}
		&\leq
		\norm{(u_1(x)-u_1(\bar x))\cdot u_2(x)}+
		\norm{u_1(\bar x)\cdot(u_2(x)-u_2(\bar x))}\\
		&\leq
		C\norm{u_1(x)-u_1(\bar x)}\norm{u_2(x)}+
		C\norm{u_1(\bar x)}\norm{u_2(x)-u_2(\bar x)}\\
		&\leq
		C\bigl(\norm{u_2}_0[u_1]_\alpha+\norm{u_1}_0[u_2]_\alpha\bigr)d(x,\bar x)^\alpha
		\fall x,\bar x\in\Omega
	\end{align*}
	and this implies that $u_1\cdot u_2$ is $\alpha$-H\"older. Then the estimate for $\norm{u_1\cdot u_2}_\alpha$ follows easily. 
\end{proof}

\begin{thm}[chain rule]
	Let $\alpha_1,\alpha_2\in[0,1]$. If functions $u_1:\Omega\to Y_1$ is $\alpha_1$-H\"older and $u_2:u_1(\Omega)\to Y$ is $\alpha_2$-H\"older, then the composition $u_2\circ u_1:\Omega\to Y$ is $\alpha_1\alpha_2$-H\"older. In case $\alpha_1,\alpha_2\in(0,1]$ one has the estimate
	$
		[u_2\circ u_1]_{\alpha_1\alpha_2}
		\leq
		[u_1]_{\alpha_1}^{\alpha_2}[u_2]_\beta
	$
	and for globally bounded $u_2$ results
	$
		\norm{u_2\circ u_1}_{\alpha_1\alpha_2}
		\leq
		\max\set{1,[u_1]_{\alpha_1}^{\alpha_2}}\norm{u_2}_{\alpha_2}. 
	$
\end{thm}
\begin{proof}
	We focus on the situation $\alpha_1,\alpha_2\in(0,1]$. The following holds
	\begin{align*}
		\norm{u_2\circ u_1(x)-u_2\circ u_1(\bar x)}
		&\leq
		[u_2]_{\alpha_2}\abs{u_1(x)-u_1(\bar x)}^{\alpha_2}
		\leq
		[u_1]_{\alpha_1}^{\alpha_2}[u_2]_{\alpha_2}d(x,\bar x)^{\alpha_1\alpha_2}
		\fall x,\bar x\in\Omega
	\end{align*}
	and so the composition $u_2\circ u_1$ is $\alpha$-H\"older. The remaining norm estimate is readily derived. 
\end{proof}

Let $B(\Omega,Y)$ abbreviate the space of globally bounded functions and we define the space of globally bounded $\alpha$-H\"older functions by
\begin{displaymath}
	C^\alpha(\Omega,Y):=\set{u\in B(\Omega,Y):\,u\text{ is $\alpha$-H\"older}}.
\end{displaymath}
Due to \tref{thmA2} the characterization $C^\alpha(\Omega,Y)=\set{u:\Omega\to Y\mid u\text{ is $\alpha$-H\"older}}$ holds on bounded spaces $\Omega$. By \lref{lemA1} it is $C^\alpha(U,Y)=C^\alpha(\overline{U},Y)$ for subsets $U\subseteq\Omega$ and Banach spaces $Y$. 
\begin{thm}\label{thmcomplete}
	The set $C^\alpha(\Omega,Y)$ is a normed space over $\K$ w.r.t.\ the norm $\norm{\cdot}_\alpha$. Furthermore, with $Y$ also $C^\alpha(\Omega,Y)$ is a Banach space. 
\end{thm}
\begin{proof}
	We merely show the completeness of $C^\alpha(\Omega,Y)$ w.r.t.\ the norm $\norm{\cdot}_\alpha$ for $\alpha\in(0,1]$. Thereto, let $(u_l)_{l\in\N}$ be a Cauchy sequence in $C^\alpha(\Omega,Y)$. Since $C^0(\Omega,Y)$ is complete in the $\sup$-norm, $(u_l)_{l\in\N}$ converges to a continuous function $u:\Omega\to Y$. It remains to show that $\lim_{l\to\infty}[u_l-u]_\alpha=0$ and that $u$ is $\alpha$-H\"older. Thereto, for $\eps>0$ first choose $L\in\N$ such that $\norm{u_l-u_m}_\alpha\leq\tfrac{\eps}{3}$ for all $l,m\geq L$, and given $x,\bar x\in\Omega$, $x\neq\bar x$, choose a fixed $\bar l\geq L$ such that
	$\norm{u_{\bar l}(x)-u(x)}\leq\tfrac{\eps}{3d(x,\bar x)^\alpha}$ and 
	$\norm{u_{\bar l}(\bar x)-u(\bar x)}\leq\tfrac{\eps}{3d(x,\bar x)^\alpha}$. Now this results in
	\begin{align*}
		\frac{\norm{(u_l-u)(x)-(u_l-u)(\bar x)}}{d(x,\bar x)^\alpha}
		&\leq
		\frac{\norm{(u_l-u_{\bar l})(x)-(u_l-u_{\bar l})(\bar x)}}{d(x,\bar x)^\alpha}+
		\frac{\norm{u_{\bar l}(x)-u(x)}}{d(x,\bar x)^\alpha}+
		\frac{\norm{u_{\bar l}(\bar x)-u(\bar x)}}{d(x,\bar x)^\alpha}\\
		&\leq
		[u_l-u_{\bar n}]_\alpha+\tfrac{2\eps}{3}
		\leq
		\eps\fall l\geq L
	\end{align*}
	and therefore $[u_l-u]_\alpha\leq\eps$. If we set $\eps=1$ in the above inequality and note that $([u_l]_\alpha)_{l\in\N}$ is bounded, then the generalized triangle inequality guarantees
	\begin{align*}
		\frac{\norm{u(x)-u(\bar x)}}{d(x,\bar x)^\alpha}
		\leq
		1+
		\frac{\norm{u_l(x)-u_l(\bar x)}}{d(x,\bar x)^\alpha}
		\leq
		1+[u_l]_\alpha
		\leq
		1+\sup_{l\in\N}[u_l]_\alpha\fall x,\bar x\in\Omega
	\end{align*}
	and consequently $u\in C^\alpha(\Omega,Y)$. 
\end{proof}
\begin{rem}\label{remark1}
	(1) The positive homogeneity $[\lambda u]_\alpha=\abs{\lambda}[u]_\alpha$ for $\lambda\in\K$ and \tref{thmsum} guarantee that $[\cdot]_\alpha$ defines a semi-norm on $C^\alpha(\Omega,Y)$. It is not a norm, since $[\cdot]_\alpha$ vanishes on the constant functions. However, if $\alpha\in(0,1]$ and $x_0\in\Omega$ is fixed, then 
	$
		\norm{u}_\alpha':=\max\set{\abs{u(x_0)},[u]_\alpha}
	$
	defines an equivalent norm on $C^\alpha(\Omega,Y)$. Indeed, any globally bounded $\alpha$-H\"older function $u:\Omega\to Y$ satisfies \eqref{appstar} for all $x\in\Omega$, which implies the inequality $\norm{u}_0\leq\norm{u(x_0)}+[u]_\alpha(\diam\Omega)^\alpha$ and consequently
	\begin{align*}
		\norm{u}_\alpha'
		&\leq
		\norm{u}_\alpha
		\leq
		\max\set{\norm{u(x_0)}+[u]_\alpha(\diam\Omega)^\alpha,[u]_\alpha}\\
		&\leq
		\max\set{\norm{u(x_0)},[u]_\alpha}
		+
		\max\set{(\diam\Omega)^\alpha,1}[u]_\alpha
		\leq
		(1+\max\set{(\diam\Omega)^\alpha,1})\norm{u}_\alpha'
	\end{align*}
	guarantees that both norms are equivalent. 

	(2) Let $\Omega$ be compact. Then $C^\alpha(\Omega)_+:=\set{u\in C^\alpha(\Omega,\R):\,0\leq u(x)\text{ for all }x\in\Omega}$ is an order cone in $C^\alpha(\Omega)$ with nonempty interior. However, if $\Omega$ has at least one accumulation point, then $C^\alpha(\Omega)_+$ is not normal (see \cite{amann:76}). 
\end{rem}

The following result establishes that $\alpha$-H\"older functions on bounded sets form a decreasing scale of spaces between the Lipschitz continuous and the uniformly continuous functions (cf.~\tref{thm1}).
\begin{lem}\label{lemembed}
	Let $\Omega$ be bounded. If $0\leq\alpha\leq\beta\leq 1$, then $\beta$-H\"older functions are $\alpha$-H\"older and satisfy $$[u]_\alpha\leq(\diam\Omega)^{\beta-\alpha}[u]_\beta$$. 
\end{lem}
\begin{proof}
	Given $u\in C^\beta(\Omega,Y)$ one has
	$$
		\norm{u(x)-u(\bar x)}
		\leq
		d(x,\bar x)^{\beta-\alpha}[u]_\beta d(x,\bar x)^\alpha
		\leq
		(\diam\Omega)^{\beta-\alpha}[u]_\beta d(x,\bar x)^\alpha
		\fall x,\bar x\in\Omega		
	$$
	yielding the assertion. 
\end{proof}
\begin{thm}[continuous embeddings]\label{thmembed}
	Let $\Omega$ be bounded. If $0\leq\alpha\leq\beta\leq 1$, then $C^\beta(\Omega,Y)\subseteq C^\alpha(\Omega,Y)$ is a continuous embedding with
	$
		\norm{u}_\alpha
		\leq
		\max\set{1,(\diam\Omega)^{\beta-\alpha}}\norm{u}_\beta
	$
	for all $u\in C^\beta(\Omega,Y)$. 
\end{thm}
In order words, one has a bounded embedding operator
\begin{equation}
	\sI_\beta^\alpha:C^\beta(\Omega,Y)\to C^\alpha(\Omega,Y)
	\fall 0\leq\alpha\leq\beta\leq 1.
	\label{noe}
\end{equation}
\begin{proof}
	If $u\in C^\beta(\Omega,Y)$, then \lref{lemembed} implies $[u]_\alpha<\infty$ and thus
	$$
		\max\set{\norm{u}_0,[u]_\alpha}
		\leq
		\max\set{\norm{u}_0,(\diam\Omega)^{\beta-\alpha}[u]_\beta}
		\leq
		\max\set{1,(\diam\Omega)^{\beta-\alpha}}\norm{u}_\beta
	$$
	yields that $u$ is also $\alpha$-H\"older and satisfies the claimed estimate. 
\end{proof}
\begin{rem}[differentiable and Sobolev functions]
	Let $\Omega\subset\R^\kappa$ be open and bounded.
	
	(1) If $\Omega$ is convex and $\bar C^1(\Omega,Y)$ denotes the (canonically normed) space of continuously differentiable functions allowing a continuous extension to $\overline{\Omega}$, then one has the continuous embedding \cite[pp.~11--12, 1.34~Thm.]{adams:fournier:02}
	\begin{equation}
		\bar C^1(\Omega,Y)\subseteq C^\alpha(\overline{\Omega},Y)\fall 0\leq\alpha\leq 1.
		\label{embed1}
	\end{equation}

	(2)	Let $\dim Y<\infty$. If $\Omega$ has a Lipschitz boundary and $k\in\N$, $p\geq 1$ satisfy $(k-\alpha)p\geq\kappa$, then the Sobolev space $W^{k,p}(\Omega,Y)$ satisfies the following continuous embedding \cite{adams:fournier:02} 
	\begin{equation}
		W^{k,p}(\Omega,Y)\subseteq C^\alpha(\overline{\Omega},Y)\fall 0<\alpha\leq 1.
		\label{embed2}
	\end{equation}
\end{rem}

For H\"older exponents $0<\alpha<\beta\leq 1$ the inclusion from Thm.~\ref{thm10} will typically be strict. 
\begin{ex}
	Let $\Omega=[0,1]$ and $Y=\R$. In case $\alpha\in(0,1)$ the function $u:[0,1]\to\R$, $u(x):=x^\alpha$ is contained in $C^\alpha[0,1]$. Now if $u$ would be $\beta$-H\"older with exponent $\beta>\alpha$, then there exists a $C\geq 0$ such that $x^\alpha=\abs{u(x)-u(0)}\leq C\abs{x-0}^\beta=Cx^\beta$ for $x\in(0,1]$ yielding the contradiction $C\geq x^{\alpha-\beta}\xrightarrow[x\searrow 0]{}\infty$. Concerning the case $\alpha=0$ we refer to \eref{exnothoelder} for a continuous function not being H\"older. 
\end{ex}

On the space $C^\alpha(\Omega,Y)$ exist several measures of noncompactness, which even are not necessarily equivalent (cf.~\cite[Sect.~5]{paret:nussbaum:11}). Among them, and for finite-dimensional spaces $Y$, we use 
\begin{displaymath}
	\chi(B):=\lim_{\eps\searrow 0}\sup_{u\in B}\set{\tfrac{\norm{u(x)-u(\bar x)}}{d(x,\bar x)^\alpha}:\,0<d(x,\bar x)\leq\eps}
\end{displaymath}
(see \cite{banas:nalepa:13,banas:nalepa:16} and \cite{saiedinezhad:19}) and obtain a sufficient compactness criterion: 
\begin{thm}[compactness in $C^\alpha(\Omega,Y)$, {cf.~\cite[Thm.~4]{banas:nalepa:13}}]\label{thmaa}
	Let $\Omega$ be compact and $\dim Y<\infty$. A subset $B\subseteq C^\alpha(\Omega,Y)$ is relatively compact, provided the following two conditions hold: 
	\begin{itemize}
		\item[(i)] $B$ is bounded,

		\item[(ii)] for every $\eps>0$ there exists a $\delta>0$ such that for all $x,\bar x\in\Omega$ one has the implication
		\begin{displaymath}
			d(x,\bar x)\leq\delta
			\quad\Rightarrow\quad
			\norm{u(x)-u(\bar x)}\leq\eps d(x,\bar x)^\alpha\fall u\in B.
		\end{displaymath}
	\end{itemize}
\end{thm}
For $\alpha=0$ this is essentially the sufficiency part of the Arzel{\`a}-Ascoli theorem \cite[p.~31, Thm.~3.2]{martin:76}, which establishes that (i) and (ii) characterize the relatively compact subsets of $C^0(\Omega,Y)$. 

\begin{thm}[compact embeddings]\label{thm10}
	Let $\Omega$ be compact and $\dim Y<\infty$. If $0\leq\alpha<\beta\leq 1$, then $C^\beta(\Omega,Y)\subseteq C^\alpha(\Omega,Y)$ is a compact embedding. Moreover, the embedding $C^\beta(\Omega,Y)\subseteq C^0(\Omega,Y)$ is even dense, provided one has $\Omega\subset\R^\kappa$. 
\end{thm}
This means that bounded subsets of $C^\beta(\Omega,Y)$ are relatively compact in $C^\alpha(\Omega,Y)$. In case $\alpha\in(0,1)$ the embedding \eqref{embed1} is compact. Similarly, for $(k-\alpha)p>\kappa$ also \eqref{embed2} is compact (see \cite[pp.~11--12, 1.34~Thm.]{adams:fournier:02}). 
\begin{proof}
	(I) Let $B\subseteq C^\beta(\Omega,Y)$ be bounded, that is, there exists a $C\geq 0$ such that $\norm{u}_\beta\leq C$ for all $u\in B$. This implies $\norm{u(x)}\leq C$ and
	$
		\norm{u(x)-u(\bar x)}
		\leq
		Cd(x,\bar x)^\beta
		\leq
		Cd(x,\bar x)^{\beta-\alpha}d(x,\bar x)^\alpha
	$
	for all $x,\bar x\in\Omega$ and $u\in B$, which guarantees that $B\subseteq C^\beta(\Omega,Y)$ satisfies the assumptions of \tref{thmaa}. Consequently, $B$ is a relatively compact subset of $C^\alpha(\Omega,Y)$. 
	
	(II) Referring to the Stone-Weierstra{\ss} theorem \cite[p.~218, Thm.~16.1]{dibenedetto:16} the polynomials over a compact $\Omega\subset\R^\kappa$ form a set of $\beta$-H\"older functions being dense in the continuous functions. 
\end{proof}
However, note that the embedding $C^\beta(\Omega,Y)\subseteq C^\alpha(\Omega,Y)$ is not dense for $0<\alpha<\beta\leq 1$. 
\begin{ex}
	Let $\Omega=[0,1]$, $Y=\R$ and $u\in C^\alpha[0,1]$ be given by $u(x):=x^\alpha$. Choose $v\in C^\beta[0,1]$ and consider the difference $u-v\in C^\alpha[0,1]$ satisfying
	\begin{displaymath}
		\frac{\abs{(u-v)(x)-(u-v)(0)}}{\abs{x-0}^\alpha}
		\geq
		\frac{\abs{u(x)-u(0)}}{\abs{x-0}^\alpha}
		-
		\frac{\abs{v(x)-v(0)}}{\abs{x-0}^\alpha}
		=
		1
		-
		\frac{\abs{x}^\beta}{\abs{x}^\alpha}
		\frac{\abs{v(x)-v(0)}}{\abs{x-0}^\beta}
		\xrightarrow[x\searrow 0]{}
		1.
	\end{displaymath}
	This implies that any function $v\in C^\beta[0,1]$ has $\alpha$-norm greater or equal to $1$ from $u$. 
\end{ex}

The final example demonstrates that $C^\alpha(\Omega,Y)$ is not separable. 
\begin{ex}
	Let $\Omega=[0,1]$, $Y=\R$ and for reals $c\in(0,1)$ define the $\alpha$-H\"older functions
	\begin{align*}
		u_c:[0,1]&\to\R,&
		u_c(x)&:=
		\begin{cases}
			0,&0\leq x\leq c,\\
			(x-c)^\alpha,&a<x\leq 1,
		\end{cases}
	\end{align*}
	where $\alpha\in(0,1]$. This implies the inequality
	\begin{align*}
		\norm{u_a-u_b}_\alpha
		&\geq
		[u_a-u_b]_\alpha
		=
		\sup_{\substack{x,\bar x\in[0,1]\\ x\neq\bar x}}\frac{(u_a-u_b)(x)-(u_a-u_b)(\bar x)}{\abs{x-\bar x}^\alpha}
		\geq
		\frac{\abs{(u_a-u_b)(b)-(u_a-u_b)(a)}}{\abs{b-a}^\alpha}\\
		&=
		\frac{(b-a)^\alpha}{\abs{b-a}^\alpha}
		=
		1\fall a,b\in(0,1)
	\end{align*}
	with the uncountable family $\set{u_c}_{c\in(0,1)}\subseteq C^\alpha[0,1]$. 
\end{ex}
\end{appendix}
\providecommand{\bysame}{\leavevmode\hbox to3em{\hrulefill}\thinspace}

\end{document}